\documentclass[11pt]{amsart}

\usepackage{epigamath}

\usepackage[english]{babel}

\numberwithin{equation}{section}

\usepackage[shortlabels]{enumitem}
\setlist[enumerate,1]{label={\rm(\arabic*)}, ref={\rm\arabic*}}

\usepackage{amssymb, amsmath, amsthm}
\usepackage{hyperref}
\usepackage{verbatim}
\usepackage{calrsfs}
\usepackage{todonotes}
\usepackage{tikz}
\usetikzlibrary{matrix,arrows,calc,cd}
\usepackage{package}
\newcommand{\lra}{\longrightarrow}
\newcommand{\supth}[1]{\ensuremath{#1^{\mathrm{th}}}}

\EpigaVolumeYear{7}{2023} \EpigaArticleNr{18} \ReceivedOn{September 22, 2022}
\InFinalFormOn{March 6, 2023}
\AcceptedOn{June 6, 2023}

\title{Torus actions on quotients of affine spaces}
\titlemark{Torus actions on quotients of affine spaces}

\author{Ana-Maria Brecan}
\email{ana-maria.brecan@uni-bayreuth.de}
\address{University of Bayreuth, Faculty of Mathematics, Physics, and Computer Science, Universit\"ats\-stra{\ss}e 30, 95447 Bayreuth, Germany}
\author{Hans Franzen}
\email{hans.franzen@math.upb.de}
\address{Paderborn University, Institute of Mathematics, 
Warburger Stra{\ss}e 100, 
33098 Paderborn, Germany}

\authormark{A.-M.~Brecan and H.~Franzen} 

\AbstractInEnglish{We study the locus of fixed points of a torus action on a GIT quotient of a complex vector space by a reductive complex algebraic group which acts linearly. We show that, under the assumption that $G$ acts freely on the stable locus, the components of the fixed point locus are again GIT quotients of linear subspaces by Levi subgroups.}



\acknowledgement{While this research was conducted, H.F.\ was supported by the DFG SFB TR 191 ``Symplektische Strukturen in Geometrie, Algebra und Dynamik''.}

\begin{document}


\maketitle

\begin{prelims}

\DisplayAbstractInEnglish

\bigskip

\DisplayKeyWords

\medskip

\DisplayMSCclass

\end{prelims}


\newpage

\setcounter{tocdepth}{1}

\tableofcontents


\section{Introduction}

	In this paper, we are concerned with actions of tori on geometric quotients of linear actions of reductive complex algebraic groups on vector spaces. Such geometric quotients appear very often. One prominent class of examples are moduli spaces of quivers without relations; see \textit{e.g.}\ \cite{Reineke:08} for an overview of this theory. Geometric quotients of vector spaces arise in other contexts as well. For example, every toric variety whose fan satisfies some mild hypotheses can be realized as such a quotient; see \cite[Theorem~2.1]{Cox:97}. Other instances where GIT quotients of vector spaces are investigated are \cite{ES:89, Halic:10, Hoskins:14}.
	
	We study actions of a torus $\TT$ on a GIT quotient of the form $V^{\st}(G,\theta)/G$, where $G$ is a reductive complex algebraic group, $V$ is a finite-dimensional representation of $G$, and $\theta$ is a character of $G$. We assume that this torus action comes from a linear $\TT$-action on $V$ that commutes with the action of $G$. Our main objective in this paper is to determine the locus of fixed points of this torus action.
	
	In the case of quiver moduli, this question has been studied by Weist in \cite{Weist:13}. More precisely, he considers torus actions on moduli spaces of stable quiver representations which are given by scaling of the maps along the arrows of the quiver. He shows that the fixed point locus of such a torus action decomposes into irreducible components, each of which is isomorphic to a moduli space of stable representations of a covering quiver.
	
	We generalize this result to GIT quotients of the form $V^{\st}(G,\theta)/G$. We need to make the assumption that the action of $G$ on the stable locus is free. This is automatically satisfied in the context of quiver moduli spaces. 
	
	Our main result, Theorem~\ref{t:main}, asserts that, under the assumption from the previous paragraph, the fixed point locus of the action of a torus on $V^{\st}(G,\theta)/G$ decomposes into irreducible components, each of which is again a stable GIT quotient of a linear subspace of $V$ by a Levi subgroup of $G$. More precisely, we fix a maximal torus $T$ in $G$ and let $W$ be its Weyl group. For every morphism $\rho\colon \TT \to T$ of algebraic groups, we set $V_\rho = \{v \in V \mid t.v = \rho(t)v \text{ (all $t \in \TT$)}\}$ and define $G_\rho$ to be the centralizer in $G$ of $\im \rho$.
	
	\begin{thm*}[{Theorem~\ref{t:main}}]
		The $\TT\!$-fixed point locus $(V^{\st}(G,\theta)/G)^{\TT}$ decomposes into connected components $\smash{\bigsqcup_{\rho}} \smash{F_\rho}$ indexed by a full set of representatives of morphisms of tori $\rho\colon \TT \to T$ up to conjugation with $W$. The component $F_\rho$ equals the image of\, $V_\rho \cap V^{\st}(G,\theta)$ in $V^{\st}(G,\theta)/G$, it is irreducible, and it is isomorphic to 
		\[
			V_\rho^{\st}(G_\rho,\theta)/G_\rho.
		\]
	\end{thm*}
	
	The index set of this disjoint union is infinite, but there are only finitely many $\rho$ for which $F_\rho$ is non-empty. This is discussed in Section~\ref{s:finiteness}.
	
	There are two main ingredients of the proof of the above main result. The first is concerned with the description of the intersection of the (semi\nobreakdash-)\-stable locus with $V_\rho$.
        
	\begin{thm*}[Theorem~\ref{t:stab_subgroup}]
		Let $\rho\colon \TT \to G$ be a morphism of algebraic groups. Then
		\begin{enumerate}
			\item $V_\rho \cap V^{\sst}(G,\theta) = V_\rho^{\sst}(G_\rho,\theta)$;
			\item $V_\rho \cap V^{\st}(G,\theta) = V_\rho^{\st}(G_\rho,\theta)$.
		\end{enumerate}
	\end{thm*}
	
	The proof of the first statement uses Kempf's theory of optimal destabilizing one-parameter subgroups \cite{Kempf:78}. For the second statement, we need to define, associated with a one-parameter subgroup $\lambda$ of $G$ and a vector $v \in V$, a Zariski-closed subset $Y_{[\lambda]}(v)$ of $G/P_\lambda$ and find a $\TT\!$-fixed point of it. The definition of $Y_{[\lambda]}(v)$ can be found in Section~\ref{s:Opt}.
	
	The second major auxiliary result is Theorem~\ref{t:closed_immersion}. 
	
	\begin{thm*}[Theorem~\ref{t:closed_immersion}]
		The induced morphism $i_\rho\colon  V_\rho^{\st}(G_\rho,\theta)/G_\rho \to V^{\st}(G,\theta)/G$ is a closed immersion.
	\end{thm*}
	
	To show this, it is enough to prove that $i_\rho$ is proper, injective (on $\C$-valued points) and that it induces an injective map on tangent spaces. We show all three assertions in Section~\ref{s:embedding}. Properness is shown using a result of Luna.
	
	Throughout the paper, we work over the field of complex numbers. All results hold over an algebraically closed field of characteristic zero, though. In positive characteristic, Lemma~\ref{l:morphism}, which is used to assign to every component of the fixed point locus a morphism  $\rho\colon  \mathcal{T} \to G$ of algebraic groups, fails; see Example~\ref{e:positive_char}. Therefore, we do not know how to describe the fixed point locus over a field of positive characteristic.
	
	The paper is organized as follows. In Section~\ref{s:setup}, we fix the setup. Section~\ref{s:Opt} recalls Kempf's theory of optimal destabilizing one-parameter subgroups; we define the sets $Y_{[\lambda]}(v)$. In Section~\ref{s:lifts}, we show that lifts of fixed points induce a morphism $\rho\colon \TT \to G$, and we prove Theorem~\ref{t:stab_subgroup}. Section~\ref{s:embedding} is devoted to the proof of Theorem~\ref{t:closed_immersion}. In Section~\ref{s:fixedPointLocus}, the main result is stated and proved. Section~\ref{s:finiteness} is concerned with finiteness conditions on the index set of the irreducible components of the fixed point locus.
	In Sections~\ref{s:quiver_moduli} and~\ref{s:toric}, we apply our theory to some classes of quotients. We discuss quiver moduli and show how to derive Weist's result \cite[Theorem~3.8]{Weist:13} from ours in Section~\ref{s:quiver_moduli}. In Section~\ref{s:toric}, we discuss the case where $G$ is a torus and compare our description of the fixed point locus with the characterization in terms of the toric fan.
	
\subsection*{Acknowledgments}
We are grateful to Michel Brion, Alan Huckleberry, and Markus Reineke for interesting conversations and insightful remarks on this subject. We would like to thank Sergey Mozgovoy, who suggested an argument that made it possible to get rid of an extra assumption. We thank the referee for various helpful comments that improved the exposition and for interesting questions.

	\section{Setup} \label{s:setup}
	
	Let us introduce the setup with which we are going to work throughout the article. We are interested in GIT quotients of complex vector spaces by actions of reductive complex algebraic groups. As references for actions of algebraic groups and (geometric) invariant theory, the reader may consult, for instance, \cite{Brion:10, Dolgachev:03, GIT:94}.
	
	We work over the complex numbers. Let $G$ be a connected reductive algebraic group.	
	Let $V$ be a finite-dimensional representation of $G$. For the action of an element $g \in G$ on a vector $v \in V$, we write $gv$. We consider $V$ as an affine space, and we are interested in GIT quotients for this action. Let $X^*(G)$ be the group of characters of $G$. 
	For $\chi \in X^*(G)$, we define
	\[
		\C[V]^{G,\chi} = \{ f \in \C[V] \mid f(gv) = \chi(g)f(v) \text{ all $g \in G$, $v \in V$}\}
	\] 
	and call $f \in \C[V]^{G,\chi}$ a $\chi$-semi-invariant function.
	Fix $\theta \in X^*(G)$.
	
	\begin{defn}
		Let $v \in V$.
		\begin{enumerate}
			\item We call $v$ a $\theta$-\emph{semi-stable} point if there exist $n>0$ and $f \in \C[V]^{G,n\theta}$ such that $f(v) \neq 0$. Let $V^{\sst}(G,\theta) \sub V$ be the subset of $\theta$-semi-stable points.
			\item The point $v$ is called $\theta$-\emph{stable} if $v$ is $\theta$-semi-stable and has finite stabilizer and $G\cdot v$ is closed in $V^{\sst}(G,\theta)$. Let $V^{\st}(G,\theta)$ be the set of $\theta$-stable points.
			\item We call $v$ a $\theta$-\emph{unstable} point if $v$ is not $\theta$-semi-stable. Let $V^{\unst}(G,\theta)$ be the set of $\theta$-unstable points.
		\end{enumerate}
	\end{defn}
	
	\begin{rem}
		For a character $\chi \in X^*(G)$, we define the $G$-linearized line bundle $L(\chi)$ as the trivial line bundle together with the $G$-action on the total space $G \times V \times \C \to V \times \C$ given by $g\cdot (v,z) = (gv,\chi(g)z)$. Then
		\[
			H^0(V,L(\chi))^G = \C[V]^{G,\chi}; 
		\]
		\textit{i.e.}\ $G$-invariant sections of $L(\chi)$ are the same as $\chi$-semi-invariant functions.  This shows that a point $v \in V$ is $\theta$-semi-stable if and only if it is semi-stable with respect to $L(\theta)$ in the sense of Mumford (see \cite[Definition~1.7]{GIT:94}). Our notion of $\theta$-stability agrees with proper stability with respect to $L(\theta)$, as defined in \cite[Definition~1.8]{GIT:94}.
	\end{rem}

	We define the graded ring $\C[V]_\theta^G := \bigoplus_{n \geq 0} \C[V]^{G,n\theta}$. We obtain a good categorical quotient
	\[
		\pi\colon  V^{\sst}(G,\theta) \lra V^{\sst}(G,\theta)/\!\!/G := V/\!\!/_{L(\theta)}G = \Proj\left(\C[V]_\theta^G\right).
	\]
	We also define $V^{\st}(G,\theta)/G := \pi(V^{\st}(G,\theta))$. The restriction of $\pi$ gives a geometric quotient.
	
	Let us recall the Hilbert--Mumford criterion for (semi\nobreakdash-)stability from \cite{GIT:94}. We are using the form stated in King's paper \cite[Proposition~2.5]{King:94}.
	
	\begin{thm}[Hilbert--Mumford criterion]
		Let $v \in V$.
		\begin{enumerate}
			\item The point $v$ is $\theta$-semi-stable if and only if $\langle \theta,\eta \rangle \geq 0$ for every one-parameter subgroup $\eta\colon  \C^\times \to G$ for which $\lim_{z \to 0} \eta(z)v$ exists.
			\item The point $v$ is $\theta$-stable if and only if $\langle \theta,\eta \rangle > 0$ for every non-trivial one-parameter subgroup $\eta\colon  \C^\times \to G$ for which $\lim_{z \to 0} \eta(z)v$ exists.
		\end{enumerate}
	\end{thm}

	We will assume the following to hold throughout the text.
		
	\begin{ass} \label{a:general_assumption}
		We suppose that $G$ acts freely on $V^{\st}(G,\theta)$.
	\end{ass}
	
	Assumption~\ref{a:general_assumption} guarantees that $\pi\colon  V^{\st}(G,\theta) \to V^{\st}(G,\theta)/G$ is a principal $G$-bundle in the {\'e}tale topology. This follows from Luna's slice theorem \cite[Section~III.1]{Luna:73}. \smash{\'Etale} local triviality implies smoothness of $V^{\st}(G,\theta)/G$.
	
	Consider a torus $\TT := (\C^\times)^r$. Suppose that $\TT$ acts linearly on $V$ in such a way that the actions of $G$ and $\TT$ commute.
	
	\begin{lem}
		The subsets $V^{\st}(G,\theta)$ and $V^{\sst}(G,\theta)$ are $\TT\!$-invariant.
	\end{lem}
	
	\begin{proof}
		Both assertions are easily proved using the Hilbert--Mumford criterion. We will only show the $\TT\!$-invariance of the stable locus. The invariance of the semi-stable locus is entirely analogous.
		
		Let $v \in V^{\st}(G,\theta)$ and $t \in \TT$. Suppose that $\eta\colon  \C^\times \to G$ is a one-parameter subgroup for which $\lim_{z \to 0} \eta(z)t.v$ exists. As the actions of $\TT$ and $G$ commute, we get $\eta(z)t.v = t.\eta(z)v$, and hence
		\[
			\lim_{z \to 0}\eta(z)v = t^{-1}.\lim_{z\to 0}\eta(z)t.v
		\]
		exists. The stability of $v$ implies $\langle \theta,\eta \rangle > 0$.
	\end{proof}
	
	This implies that the $\TT\!$-action descends to an action on the quotient $V^{\sst}(G,\theta)/\!\!/G$ and to $V^{\st}(G,\theta)/G$. We are going to describe the fixed point locus of the latter action.

	\begin{rem} \label{r:normal_form}
		As we work in characteristic zero, the group $G$ is linearly reductive, so the finite-dimensional representation $V$ decomposes into irreducible representations, say $V \cong V_1^{n_1} \oplus \ldots \oplus V_l^{n_l}$, where $V_1,\ldots,V_l$ are pairwise non-isomorphic irreducible representations of $G$. An application of Schur's lemma shows that $t \in \TT$ leaves the summand $V_i^{n_i}$ invariant. Moreover, as $\C$ is algebraically closed, Schur's lemma also tells us that $t$ acts on $V_i^{n_i} \cong V_i \otimes \C^{n_i}$ as an element of $\GL_{n_i}(\C)$. The thus obtained $\TT$-action on $\C^{n_i}$ can be diagonalized: we find a basis $e_{i,1},\ldots,e_{i,n_i}$ of $\C^{n_i}$ such that $\TT$ operates on $e_{i,a}$ with weight $w_{i,a} \in X^*(\TT) \cong \Z^r$. This means that every vector $v \in V$ can be decomposed in a unique way as $v = \smash{\sum_{i=1}^l} \sum_{a=1}^{n_i} v_{i,a} \otimes e_{i,a}$ with $v_{i,a} \in V_i$ and a torus element $t \in \TT$ acts on $v$ by
		\[
			t.v = \sum_{i=1}^l \sum_{a=1}^{n_i} w_{i,a}(t)v_{i,a} \otimes e_{i,a}.
		\]		
		The irreducible representations of $G$ are completely classified. Choose a maximal torus $T$ of $G$. Let $D \sub X^*(T)$ be the set of dominant weights. For every $\lambda \in D$, there exists a unique representation $M_G(\lambda)$ of $G$ whose highest weight is $\lambda$. For a weight $\mu \in X^*(T)$, the dimension $d_{\lambda,\mu}$ of the weight space $M_G(\lambda)_\mu$ is combinatorially determined as the number of ways to write $\lambda-\mu$ as a sum of positive roots.
	\end{rem}

	Let $T$ be a maximal torus of $G$. Let $\mu \in X^*(T)$. By $V_\mu(T)$ we denote the weight space of weight $\mu$. In the same vein, for $w \in X^*(\TT)$, we let $V_w(\TT)$ be the weight space of weight $w$ with respect to the $\TT\!$-action. As the actions commute, we obtain an action of $T \times \TT$. The weight space $V_{(\mu,w)}(T \times \TT)$ agrees with $V_\mu(T) \cap V_w(\TT)$.
	Let $N_T(V) = \{ \mu \in X^*(T) \mid V_\mu(T) \neq 0 \}$ be the set of weights of the $T$-action on $V$. Define $N_\TT(V)$ and $N_{T \times \TT}(V)$ in just the same way.

	\section{Optimal destabilizing one-parameter subgroups} \label{s:Opt}
	
	The theory of optimal destabilizing one-parameter subgroups is developed by Kempf in \cite{Kempf:78}; see also \cite{Kempf:18} for a transcript of a preliminary version of the article which contains simpler, more accessible versions of the main results. Kempf's theory is used by Hesselink \cite{Hesselink:78}, Kirwan \cite{Kirwan:84}, and Ness \cite{Ness:84} to give a stratification of the unstable locus. A treatment of Kempf's theory for instability with respect to a character can be found in Hoskin's paper \cite{Hoskins:14}. We are going to recall the basics of this theory for the convenience of the reader.

	Assume that $v$ is unstable with respect to the $G$-action and the stability condition $\theta$. Then, by the Hilbert--Mumford criterion, there exists a one-parameter subgroup $\eta$ of $G$ such that $\lim_{z \to 0} \eta(z)v$ exists and $\langle \theta, \eta \rangle < 0$. 
	
	Kempf shows in \cite{Kempf:78} how to find a one-parameter subgroup which is ``most responsible'' for the instability of $v$. To this end, we again fix a maximal torus $T \sub G$ and choose an inner product $(\blank,\blank)$ on the vector space $X_*(T)_\R$ which is $W$-invariant (\textit{i.e.}\ $(w\lambda w^{-1},w\mu w^{-1}) = (\lambda,\mu)$ for all $\lambda,\mu \in X_*(T)_\R$) and for which $(\lambda,\lambda) \in \Z$ provided that $\lambda \in X_*(T)$. We thus obtain a $W$-invariant norm $\|\blank\|$ on $X_*(T)_\R$ whose square takes integral values on $X_*(T)$.
	For any $\eta \in X_*(G)$, there exists a $g \in G$ such that $g\eta g^{-1} \in X_*(T)$. By the Weyl group invariance of the norm, 
	setting
	\[
		\|\eta\| := \|g\eta g^{-1}\|
	\]
	is well defined. By definition, the value of $\|\eta\|$ is invariant under conjugation by $G$.
	
	Now for any $v \in V$, we set
	\[
		m_G^\theta(v) := \inf \left\{ \frac{\langle \theta,\eta \rangle}{\|\eta\|}\ \middle|\ 1 \neq \eta \in X_*(G) \text{ such that } \lim_{z \to 0} \eta(z)v \text{ exists}\right\}.
	\]
	Obviously, $v \in V^{\unst}(G,\theta)$ if and only if $m_G^\theta(v) < 0$.
	
	\begin{defn}
		Let $v \in V^{\unst}(G,\theta)$. A one-parameter subgroup $\lambda \in X_*(G)$ is called \emph{adapted} to $v$ if
		\begin{itemize}
			\item $\lim_{z \to 0} \lambda(z)v$ exists, and
			\item $\frac{\langle \theta,\lambda \rangle}{\|\lambda\|} = m_G^\theta(v)$.
		\end{itemize}
		We call a one-parameter subgroup $\lambda$ \emph{primitive} if it is not divisible by any integer $d \geq 2$.
		Let 
		\[
			\Lambda_G^\theta(v) := \{ \lambda \in X_*(G) \mid \lambda \text{ is primitive and adapted to $v$} \}.
		\]
	\end{defn}
	
	Let $G_\lambda$ be the centralizer of the image of $\lambda$ inside $G$, and let $P_\lambda$ be the closed subgroup
	\[
		P_\lambda = \left\{ g \in G\ \middle| \ \lim_{z\to 0} \lambda(z)g\lambda(z)^{-1} \text{ exists}\right\}.
	\]
	Then $P_\lambda$ is a parabolic subgroup of $G$ with Levi factor $G_\lambda$. It can also be described as the closed subgroup whose Lie algebra $\mathfrak{p}_\lambda$ is the sum of the non-negative weight spaces $\bigoplus_{n \geq 0} \mathfrak{g}^\lambda_n$ with respect to the adjoint action of $\lambda$ on the Lie algebra $\mathfrak{g}$. The following result is \cite[Theorem~2.2]{Kempf:78}; we use the form stated in \cite[Theorem~2.15]{Hoskins:14}.

	\begin{thm}[Kempf] \label{t:Kempf}
		Let $v \in V^{\unst}(G,\theta)$. Then: 
		\begin{enumerate}
			\item We have $\Lambda_G^\theta(v) \neq \emptyset$.
			\item We have $\Lambda_G^\theta(gv) = g\Lambda_G^\theta(v)g^{-1}$. 
			\item For any two $\lambda,\lambda' \in \Lambda_G^\theta(v)$, the parabolic subgroups $P_\lambda$ and $P_{\lambda'}$ agree. We denote this subgroup by $P_v$.
			\item The subset $\Lambda_G^\theta(v) \sub X_*(G)$ is a $P_v$-conjugacy class.
			\item For every maximal torus $H \sub P_v$, there exists a unique $\lambda \in \Lambda_G^\theta(v) \cap X_*(H)$.
		\end{enumerate}
	\end{thm}
	
	\begin{rem}
		If the vector $v$ is $\theta$-semi-stable, then we have $\Lambda_G^\theta(v) \neq \emptyset$ as well, and for a given torus $H$, there is still a primitive one-parameter subgroup $\lambda$ of $H$ which is adapted to $v$, but $\lambda$ need not be unique. This is because \cite[Lemma~1]{Kempf:18} requires a certain function to attain a positive value somewhere, which is not necessarily true if $v$ is semi-stable. Hence, we cannot conclude that $P_\lambda = P_{\lambda'}$, nor will $\Lambda_G^\theta(v)$ be a conjugacy class. 
	\end{rem}
	
	For our purposes, we need some further results of this kind. Fix a one-parameter subgroup $\lambda_0 \in X_*(G)$. Denote by $P_0 := P_{\lambda_0}$ its associated parabolic subgroup of $G$. Let $[\lambda_0] := [\lambda_0]_G$ be the $G$-conjugacy class inside $X_*(G)$. Define
	\[
		f\colon  [\lambda_0] \lra G/P_0
	\]
	by $f(\lambda) = P_\lambda$. Note that we have identified the closed points of $G/P_0$ with the set of all parabolic subgroups of $G$ which are conjugate to $P_0$. The map $f$ is well defined as for $\lambda = g\lambda_0g^{-1}$, we have $f(\lambda) = P_\lambda = gP_0g^{-1}$. The above argument also shows that $f$ is $G$-equivariant and surjective.
	
	\begin{lem} \label{l:conjClass}
		The fiber of $f$ in $P \in G/P_0$ is a single $P$-conjugacy class.
	\end{lem}
	
	\begin{proof}
		Let $\lambda,\mu \in [\lambda_0]$ be such that $f(\lambda) = f(\mu) = P$. Choose $g, h \in G$ such that $\lambda = g\lambda_0g^{-1}$ and $\mu = h\lambda_0h^{-1}$. Then $P = gP_0g^{-1} = hP_0h^{-1}$ and thus $p_0 := g^{-1}h \in P_0$. We see that 
		\[
			\mu = h\lambda_0h^{-1} = gp_0\lambda_0p_0^{-1}g^{-1} = \bigl(\underbrace{gp_0g^{-1}}_{=: p \in P}\bigr)\,g\lambda_0g^{-1}\left(gp_0g^{-1}\right)^{-1} = p\lambda p^{-1}. 
		\]
		Conversely, if $f(\lambda) = P$ and $\mu = p\lambda p^{-1}$ for some $p \in P$, then it is easy to see that $f(\mu) = P$.
	\end{proof}
	
	Now let $v \in V$ and consider the set
	\[
		L_G(v) := \left\{ \lambda \in X_*(G)\ \middle| \  \lim_{z \to 0} \lambda(z)v \text{ exists} \right\}.
	\]
	This set can be described as follows. Consider $V$ as a representation $V(\lambda)$ of $\C^\times$ via $\lambda$, and consider the weight space decomposition $V(\lambda) = \bigoplus_{n \in \Z} V_n(\lambda)$, \textit{i.e.}\ 
	\[
		V_n(\lambda) := \{ w \in V \mid \lambda(z)w = z^nw \text{ (all $z \in \C^\times$)} \}.
	\] 
	Decompose the vector $v$ as $v = \sum_{n \in \Z} v_n(\lambda)$ with $v_n(\lambda) \in V_n(\lambda)$. We obtain $L_G(v) = \{\lambda \in X_*(G) \mid v_n(\lambda) = 0 \text{ (all $n<0$)}\}$.
	
	\begin{rem}
		For $\lambda \in L_G(v)$ and $p \in P_\lambda$, we have $p\lambda p^{-1} \in L_G(v)$.
	\end{rem}

	\begin{defn}
		We define $Y_{[\lambda_0]}(v) := f([\lambda_0] \cap L_G(v))$.
	\end{defn}

	\begin{prop} \label{p:closed}
		The set $Y_{[\lambda_0]}(v)$ is a Zariski-closed subset of\, $G/P_0$.
	\end{prop}

	\begin{proof}
		To show this proposition, we consider the subset
		\[
			X_{[\lambda_0]}(v) := \left\{ g \in G\ \middle| \  \lim_{z \to 0} g\lambda_0(z)g^{-1}v \text{ exists} \right\}.
		\]
		Evidently, $X_{[\lambda_0]}(v)$ is the inverse image of $Y_{[\lambda_0]}(v)$ under the quotient map $\pi\colon  G \to G/P_0$. This shows that $X_{[\lambda_0]}(v)$ is $P_0$-invariant. As $\pi$ maps $P_0$-invariant closed subsets of $G$ to closed subsets of $G/P_0$, the  proposition will follow from the lemma below.
	\end{proof}
	
	\begin{lem}
		The subset $X_{[\lambda_0]}(v)$ is Zariski closed in $G$.
	\end{lem}

	\begin{proof}
		We introduce some notation. Let $V_n := V_n(\lambda_0)$ and $V_{\geq 0} := \bigoplus_{m \geq 0} V_m$. For $g \in G$, we let 
		\[
			V_n^{(g)} := V_n\left(g\lambda_0g^{-1}\right) = gV_n
		\]
		and $\smash{v_n^{(g)}} = v_n(g\lambda_0g^{-1}) \in V_n^{(g)}$, so $v = \sum_{n \in \Z} v_n^{(g)}$. Then
		\begin{align*}
			X_{[\lambda_0]}(v) &= \left\{ g \in G \mid v_n^{(g)} = 0 \text{ (all $n<0$)} \right\} \\
			&= \left\{ g \in G \mid g^{-1}v \in V_{\geq 0} \right\}.
		\end{align*}
		This can be translated into a fiber product diagram
		\[
			\begin{tikzcd}
				X_{[\lambda_0]}(v) \arrow{r}{} \arrow{d}{} & G \arrow{d}{} &[-2em] g \arrow[mapsto]{d}{} \\
				V_{\geq 0} \arrow{r}{} & V & g^{-1}v
			\end{tikzcd}
		\]
		which shows that $X_{[\lambda_0]}(v)$ is indeed closed.
	\end{proof}

	\section{Lifts of fixed points} \label{s:lifts}
	
	Let $y \in (V^{\st}(G,\theta)/G)^{\TT}$, and let $x \in V^{\st}(G,\theta)$ be such that $\pi(x) = y$. Then for every $t \in \TT$ there exists a unique $g \in G$ such that $t.x = gx$ (recall that the $G$-action on the stable locus is assumed to be free). This yields a map $\rho = \rho_x\colon  \TT \to G$. Note that $\rho$ depends on $x$ and not just on $y$. The following argument can also be found in \cite[Lemma~3.1]{Weist:13}. We give it for completeness.
	
	\begin{lem} \label{l:morphism}
		The map $\rho$ is a morphism of algebraic groups.
	\end{lem} 
	
	\begin{proof}
		Let $H = \{ (g,t) \in G \times \TT \mid t.x = gx \}$. This is a closed subgroup of $G \times \TT$. The projection $p_2\colon  H \to \TT$ is surjective, and it is also injective by Assumption~\ref{a:general_assumption}. As we work in characteristic 0, this means that $p_2$ is an isomorphism. Let $p_1\colon  H \to G$ be the projection to the first component. Then $\rho$ is given by $\rho = p_1\circ p_2^{-1}$.
	\end{proof}
	
	\begin{ex} \label{e:positive_char}
		The following example shows that Lemma~\ref{l:morphism} fails in positive characteristic. Let $k$ be an algebraically closed field of characteristic $p > 0$. Let $V = \mathbb{A}^2$ and $G = \mathbb{G}_m$, which acts on $V$ via $g\cdot (v_1,v_2) = (g^pv_1,g^pv_2)$. For the character $\theta = \id$, the stable locus is $V^{\st}(G,\theta) = \A^2\setminus\{(0,0)\}$. The group $G$ acts freely on the stable locus because the characteristic is $p$. The quotient $V^{\st}(G,\theta)/G$ is isomorphic to $\P^1$.

		Let $\mathcal{T} = \mathbb{G}_m$ act on $V$ by $t.(v_1,v_2) = (v_1,tv_2)$. This action commutes with the $G$-action, and the fixed points of the induced action on $\P^1$ are $[1:0]$ and $[0:1]$. For the lift $x = (0,1)$, we consider the group
		\[
			H := \{ (g,t) \in G \times \mathcal{T} \mid t.x = g\cdot x \} = \left\{ (g,t) \in G \times \mathcal{T} \mid t = g^p \right\}
		\]
		and the projections $p_1\colon  H \to G$ and $p_2\colon  H \to \mathcal{T}$. We argue that $p_2$ is not an isomorphism. We observe that $p_1$ is an isomorphism; its inverse is given by $p_1^{-1}(g) = (g,g^p)$. The composition $p_2\circ p_1^{-1}\colon  G \to \mathcal{T}$ maps $g$ to~$g^p$. This is a bijective morphism which is not an isomorphism. Therefore, $p_2$ is not an isomorphism either.
	\end{ex}

	Via $\rho$, we obtain an action of $\TT$ on every $G$-representation; for a representation $W$ of $G$, we denote the representation of $\TT$ obtained via $\rho$ by ${}_\rho{W}$. We may in particular consider ${}_\rho{V}$. This new $\TT\!$-action by $\rho$ differs in general from the $\TT\!$-action on $V$ that we started with. We look at the subset
	\[
		V_\rho := \left\{v \in V \mid t.v = \rho(t)v \text{ for all } t \in \TT \right\}
	\]
	on which these two $\TT$-actions agree. The following properties of $V_\rho$ are obvious. 
	
	\begin{lem}\label{lem43}
		Let $y$, $x$, and $\rho = \rho_x$ be as above.
		\begin{enumerate}
			\item $V_\rho$ is a linear subspace of $V$.
			\item $V_\rho$ is a $\TT\!$-subrepresentation of both $V$ $($with the initial $\TT\!$-action$)$ and ${}_\rho{V}$.
			\item We have $x \in V_\rho$. 
			\item We have $y \in \pi(V_\rho \cap V^{\st}(G,\theta)) \sub (V^{\st}(G,\theta)/G)^{\TT}$. 
		\end{enumerate}
	\end{lem}
	
	The last statement of Lemma~\ref{lem43} is the reason to consider $V_\rho$ in the first place. Let us define 
	\[
		F_\rho := \pi\left(V_\rho \cap V^{\st}(G,\theta)\right). 
	\]
	We will see in the following that $F_\rho$ is actually a closed subset of $V^{\st}(G,\theta)/G$ and is moreover the connected/irreducible component of $y$ inside $(V^{\st}(G,\theta)/G)^\TT$. We will also give a more intrinsic description of $F_\rho$. For this description, it will not be important that $\rho$ comes from the lift of a fixed point.
	
	Fix a morphism of algebraic groups $\rho\colon  \TT \to G$. Let $G_\rho$ be the centralizer of $\im(\rho)$.
	
	\begin{lem}
		The subspace $V_\rho$ is $G_\rho$-invariant.
	\end{lem}
	
	\begin{proof}
		Let $v \in V_\rho$ and $g \in G_\rho$. To show that $gv \in V_\rho$, consider $t \in \TT$. We compute
		\[
			t.gv = gt.v = g\rho(t)v = \rho(t)gv, 
		\]
		which shows the claim.
	\end{proof}
	
	\begin{rem} \label{r:weights2}
		Let $T$ be a maximal torus which contains the image of $\rho$. Recall the notation from the last paragraph of Section~\ref{s:setup}. Regarding $V_\rho$ as a representation of $T$, we obtain $(V_\rho)_\mu(T) = V_\mu(T) \cap V_{\rho^*(\mu)}(\TT)$. Consequently, 
		\[
			N_T(V_\rho) = \left\{ \mu \in X^*(T) \mid (\mu,\rho^*(\mu)) \in N_{T \times \TT}(V) \right\}.
		\]
	\end{rem}

	We now move to the central result of this section.
        
{\samepage	\begin{thm} \label{t:stab_subgroup}
		Let $\rho\colon  \TT \to G$ be a morphism of algebraic groups. Then
		\begin{enumerate}
			\item\label{t:stab_subgroup-1} $V_\rho \cap V^{\sst}(G,\theta) = V_\rho^{\sst}(G_\rho,\theta)$; 
			\item\label{t:stab_subgroup-2} $V_\rho \cap V^{\st}(G,\theta) = V_\rho^{\st}(G_\rho,\theta)$.
		\end{enumerate}
  \end{thm}
  }
	
	\begin{proof}
		\eqref{t:stab_subgroup-1} Suppose that $v \in V_\rho \cap V^{\sst}(G,\theta)$. Then the Hilbert--Mumford criterion shows at once that $v$ belongs to $V_\rho^{\sst}(G_\rho,\theta)$.
	
		For the converse inclusion, we employ Kempf's theory of optimal one-parameter subgroups; see Section~\ref{s:Opt}. Let $v \in \smash{V_\rho^{\sst}(G_\rho,\theta)}$. Assume that $v \in \smash{V^{\unst}(G,\theta)}$. Choose a primitive adapted one-parameter subgroup $\lambda \in \smash{\Lambda_G^{\theta}(v)}$. Let $t \in \mathcal{T}$, and consider $\rho(t)\lambda\rho(t)^{-1} \in X_*(G)$. It is primitive,  and we get
		\begin{align*}
			\rho(t)\lambda(z)\underbrace{\rho(t)^{-1}v}_{=t^{-1}.v} = \rho(t)t^{-1}.(\lambda(z)v) \xto{}{z \to 0} \rho(t)t^{-1}.\left(\lim_{z \to 0}\lambda(z)v\right).
		\end{align*}
		Moreover,
		\[
			\frac{\langle \theta, \rho(t)\lambda\rho(t)^{-1}\rangle}{\|\rho(t)\lambda\rho(t)^{-1}\|} = \frac{\langle \theta, \lambda\rangle}{\|\lambda\|} = m_G^\theta(v),
		\]
		which shows $\rho(t)\lambda\rho(t)^{-1} \in \Lambda_G^\theta(v)$. By Theorem~\ref{t:Kempf}, there exists a $p \in P_\lambda$ such that 
		\begin{equation} \label{e:p} \tag{*}
			p\lambda p^{-1} = \rho(t)\lambda\rho(t)^{-1}.
		\end{equation}
		We re-write this identity as
		\[
			\lambda(z)\rho(t)^{-1}\lambda(z)^{-1} = \rho(t)^{-1}p\lambda(z)p^{-1}\lambda(z)^{-1}
		\]
		which, as $p^{-1}$ lies in $P_\lambda$, converges as $z \to 0$. This shows that $\rho(t)$ itself belongs to $P_\lambda$. 
		As $\rho(t)$ is a semi-simple element, it is contained in a maximal torus $T'$ of $P_\lambda$; in fact, the whole image of $\rho$ lies in some maximal torus $T'$. For the same reason, $\lambda$ is contained in a maximal torus $T'' \sub G_\lambda \sub P_\lambda$. The two maximal tori are $P_\lambda$-conjugate, so we find a $q \in P_\lambda$ such that $qT'q^{-1} = T''$.
		In particular,  $q\rho(t)q^{-1} \in T'' \sub G_\lambda$, which implies 
		\[
			q\rho(t)q^{-1}\lambda(z) = \lambda(z)q\rho(t)q^{-1}
		\]
		for all $t\in\TT$ and all $z\in \C^\times$.
		As this identity holds for every $t$, we find $q^{-1}\lambda(z)q \in G_\rho$ for every $z$. As $q \in P_\lambda$, we obtain
		\[
			\lambda' := q^{-1}\lambda q \in \Lambda_G^\theta(v) \cap X_*(G_\rho).
		\]
		This says $\lambda'$ is a one-parameter subgroup of $G_\rho$ for which the limit $\lim_{z\to0}\lambda'(z)v$ exists and $\langle \theta,\lambda' \rangle = \norm{\lambda'}m_G^\theta(v) < 0$.
		This gives a contradiction to $v \in V_\rho^{\sst}(G_\rho,\theta)$.

		\eqref{t:stab_subgroup-2} Again, as in the first item, the inclusion from left to right follows readily from the Hilbert--Mumford criterion.

		Now to the converse inclusion. Let $v \in V_\rho^{\st}(G_\rho,\theta)$, and assume $v$ is not $\theta$-stable for the $G$-action on~$V$. By~\eqref{t:stab_subgroup-1}, we know that $v$ is at least $\theta$-semi-stable for the $G$-action. So there must be a non-trivial one-parameter subgroup $\lambda \in X_*(G)$ for which $\lim_{z \to 0}\lambda(z)v$ exists---so $\lambda \in L_G(v)$---and such that $\langle \theta,\lambda \rangle = 0$. Then $\langle \theta,g\lambda g^{-1} \rangle =0$ for every $g \in G$, and as in \eqref{t:stab_subgroup-1} we see that $\rho(t)\lambda\rho(t)^{-1} \in L_G(v)$ for every $t \in \mathcal{T}$. Therefore, $\mathcal{T}$ acts via $\rho$ on $[\lambda] \cap L_G(v)$, and by the $G$-equivariance of the map $f\colon  [\lambda] \cap L_G(v) \to G/P_\lambda$, the Zariski-closed subset $Y_{[\lambda]}(v) \sub G/P_\lambda$ (see Proposition~\ref{p:closed}) is $\mathcal{T}$-invariant. By Borel's fixed point theorem, there is a fixed point for the $\mathcal{T}$-action on $Y_{[\lambda]}(v)$. Let $P$ be a $\mathcal{T}$-fixed point in $Y_{[\lambda]}(v)$. Then there exists a $\mu \in [\lambda] \cap L_G(v)$ such that $P = P_\mu$. By the fixed point property, we get
		\[
			f(\rho(t)\mu\rho(t)^{-1}) = f(\mu).
		\]
		As the fiber $f^{-1}(P_\mu)$ is a $P_\mu$ conjugacy class (see Lemma~\ref{l:conjClass}), we find a $p \in P_\mu$ such that $p\mu p^{-1} = \rho(t)\mu\rho(t)^{-1}$. Now we are in equation (\ref{e:p}) of the proof of \eqref{t:stab_subgroup-1}. We may conclude in the same way that $q^{-1}\mu q \in L_G(v) \cap X_*(G_\rho) = L_{G_\rho}(v)$ for some $q \in P_\mu$. But $\langle \theta,q^{-1}\mu q \rangle = \langle \theta,\mu \rangle = \langle \theta,\lambda \rangle = 0$,  which contradicts $v \in V_\rho^{\st}(G_\rho,\theta)$.
		\qedhere
	\end{proof}
	
	\section{An embedding of quotients} \label{s:embedding}
	
	We now look at the inclusion of the linear subspace $V_\rho$ into $V$. The intersection with $V^{\st}(G,\theta)$ yields a closed embedding $V_\rho^{\st}(G_\rho,\theta) \to V^{\st}(G,\theta)$. It is $G_\rho$-invariant. We obtain a morphism $i_\rho\colon  V_\rho^{\st}(G_\rho,\theta)/G_\rho \to V^{\st}(G,\theta)/G$. We will show the following in this section.
	
	\begin{thm} \label{t:closed_immersion}
		The induced morphism $i_\rho\colon  V_\rho^{\st}(G_\rho,\theta)/G_\rho \to V^{\st}(G,\theta)/G$ is a closed immersion.
	\end{thm}
	
	In a previous version of this paper, the above theorem was only shown under the assumption that there were no $T$-invariants for a maximal torus $T$ of $G$. Sergey Mozgovoy suggested to us an argument to get rid of this assumption. It goes along the lines of \cite[Lemma~6.6]{Mozgovoy:23} and uses the machinery of \'{e}tale slices. After we received his suggestion, we found a related but different argument which allows us to show a finiteness statement on the affine quotients, which might be of independent interest. This is Proposition~\ref{p:finiteness}. From this proposition and two other lemmas, Theorem~\ref{t:closed_immersion} will follow.
	
	We use the following criterion for closed immersions. It can be stated in greater generality. We have adapted it to our situation. A proof of (a more general version of) the criterion below is due to Osserman.
        
	\begin{prop} \label{p:osserman}
		Let $f\colon  X \to Y$ be a morphism of complex varieties. If
		\begin{enumerate}
			\item $f$ is proper,
			\item $f$ is injective on closed points, and
			\item the induced map $d_xf\colon  T_xX \to T_{f(x)}Y$ is injective for every closed point $x \in X$,
		\end{enumerate}
		then $f$ is a closed immersion.
	\end{prop}
	
	We will prove these three properties for the morphism $i_\rho$. We need some auxiliary results. The following is the main result of \cite{Luna:75}.  
	
	\begin{thm}[Luna]
		Let $X$ be an affine variety, $G$ a reductive algebraic group which acts on $X$, and $H \sub G$ a closed reductive subgroup. Then $N_G(H)$ is reductive, and $X^H/\!\!/N_G(H) \to X/\!\!/G$ is a finite morphism.
	\end{thm}

	The proof uses the machinery of \'{e}tale slices, which requires the ground field to be algebraically closed of characteristic zero. Let us mention that a version of this result exists in positive characteristic as well; see \cite[Theorem~1.1]{BGM:19}. We use Luna's result to show the following. 
	
	\begin{prop} \label{p:finiteness}
		The morphism $V_\rho/\!\!/G_\rho \to V/\!\!/G$ is finite.
	\end{prop}
	
	\begin{proof}
		Let $X = V \times \C^r$, where $r$ is the rank of $\mathcal{T}$. On $X$ we consider the action of $\bar{G} = G \times \mathcal{T}$ via
		\[
			(g,t) * (v,w) = (gt.v, tw).
		\]
		Consider the reductive subgroup $\Delta_\rho := \{ (\rho(t)^{-1},t) \mid t \in \mathcal{T} \}$ of $\bar{G}$. We see that
		$
			N_{\bar{G}}(\Delta_\rho) = G_\rho \times \TT
		$
		and $X^{\Delta_\rho} = V_\rho \times \{0\}$. The action of $G_\rho \times \TT$ on $V_\rho \times \{0\}$ is given by $(g,t) * (v,0) = (gt.v,0) = (g\rho(t)v,0)$. This shows that
		\[
			\C\left[V_\rho \times \{0\}\right]^{G_\rho \times \TT} \cong \C\left[V_\rho\right]^{G_\rho}.
		\]
		We obtain the following commutative diagram:
		\[\begin{tikzcd}
			{V_\rho/\!\!/G_\rho} & {V/\!\!/G} \\
			{\left(V_\rho \times \{0\}\right)/\!\!/\left(G_\rho \times \mathcal{T}\right)} & {(V \times \C^r)/\!\!/(G \times \mathcal{T})\rlap{.}}
			\arrow["\cong", from=1-1, to=2-1]
			\arrow[from=1-1, to=1-2]
			\arrow[from=1-2, to=2-2]
			\arrow[from=2-1, to=2-2]
		\end{tikzcd}\]
		The right-hand map comes from $V \to V \times \C^r,\ v \mapsto (v,0)$. The bottom map is a finite morphism by Luna's result, so the top map is as well.
	\end{proof}

	\begin{lem} \label{l:properness}
		The morphism $i_\rho$ is proper.
	\end{lem}
	
	\begin{proof}
		Consider the commutative diagram
		\[\begin{tikzcd}
			{V_\rho^{\sst}\left(G_\rho,\theta\right)/\!\!/G_\rho} & {V^{\sst}(G,\theta)/\!\!/G} \\
			{V_\rho/\!\!/G_\rho} & {V/\!\!/G\rlap{.}}
			\arrow[from=2-1, to=2-2]
			\arrow[from=1-1, to=1-2]
			\arrow[from=1-1, to=2-1]
			\arrow[from=1-2, to=2-2]
		\end{tikzcd}\]
		Both vertical morphisms are projective, in particular proper. The bottom map is proper by Proposition~\ref{p:finiteness}.
		Therefore, $\smash{V_\rho^{\sst}(G_\rho,\theta)/\!\!/G_\rho} \to V^{\sst}(G,\theta)/\!\!/G$ is proper. As, by Theorem~\ref{t:stab_subgroup}, $i_\rho$ arises as the base change of this morphism, $i_\rho$ is also proper. 
	\end{proof} 
	
	\begin{lem}
		The map $i_\rho$ is injective on closed points.
	\end{lem}
	
	\begin{proof}
		Let $v_1,v_2 \in V_\rho^{\st}(G_\rho,\theta)$ be such that $i_\rho(v_1) = i_\rho(v_2)$. Then there exists a $g \in G$ such that $gv_1 = v_2$. We need to show that $g \in G_\rho$, so we take $t \in \mathcal{T}$ and show $g$ commutes with $\rho(t)$. We see that 
		\[
			\rho(t)gv_1 = \rho(t)v_2 = t.v_2 = t.gv_1 = gt.v_1 = g\rho(t)v_1.
		\]
		We have used that $v_1$ and $v_2$ belong to $V_\rho$. We deduce that the commutator of $g$ and $\rho(t)$ lies in the stabilizer $\operatorname{Stab}_G(v_1)$. As $v_1$ lies in $V^{\st}(G,\theta)$ by Theorem~\ref{t:stab_subgroup} and as stable points have trivial stabilizers by Assumption~\ref{a:general_assumption}, we obtain the desired commutativity.
	\end{proof}
	
	We now consider the tangent spaces. The tangent space of a $G$-principal bundle quotient $Y = X/G$ at a point $y = \pi(x)$ may be identified as
	\[
		T_yY \cong T_xX/\im(d_e \act_G(\blank,x)), 
	\]
	where $\act_G\colon  G \times X \to X$ is the action map and $d_e\act_G(\blank,x)\colon  \mathfrak{g} \to T_xX$ is the derivative at $e$ of the map $G \to X,\ g \mapsto gx$. Here, $X = V^{\st}(G,\theta)$ is open in a vector space, so $T_xX \cong V$. As the action is linear, $V$ becomes a $\mathfrak{g}$-module via the derivative $d_e\sigma\colon  \mathfrak{g} \to \mathfrak{gl}(V)$. We see that
	\[
		((d_e\sigma)(\xi))(v) = (d_e\act_G(\blank,v))(\xi)
	\]
	for all $\xi \in \mathfrak{g}$ and $v \in V$. We write $\xi\cdot v$ for $((d_e\sigma)(\xi))(v)$. For $g \in G$, $\xi \in \mathfrak{g}$, and $v \in V$,  $g(\xi\cdot v) = \operatorname{Ad}_g(\xi)\cdot gv$ holds; see for instance \cite[Section~II.2.3]{Kraft:84}.

	\begin{lem}
		Let $v \in V_\rho^{\st}(G_\rho,\theta)$. Then the induced map 
		\[
			d_{\pi_\rho(v)}i_\rho\colon  T_{\pi_\rho(v)}(V_\rho^{\st}(G_\rho,\theta)/G_\rho) \lra T_{\pi(v)}(V^{\st}(G,\theta)/G)
		\]
		is injective.
	\end{lem}

	\begin{proof}
		Abbreviate $Y := V^{\st}(G,\theta)/G$ and $Y_\rho := V_\rho^{\st}(G_\rho,\theta)/G_\rho$. Let $x := \pi_\rho(v)$ and $y := \pi(v) = i_\rho(x)$.
		We have already argued that
		\begin{align*}
			T_yY &\cong V/\im(d_e\act_G(\blank,v)), \\
			T_xY_\rho &\cong V_\rho/\im(d_e\act_{G_\rho}(\blank,v)).
		\end{align*}
		Let $K$ be the kernel of $d_xi_\rho$. We have the following commutative diagram with exact rows and columns:
		\[
		\begin{tikzcd}
			  &   &[3em]   & 0 \arrow{d}{} & \\
			  & 0 \arrow{d}{} & 0 \arrow{d}{} & K \arrow{d}{} & \\
			0 \arrow{r}{} & \mathfrak{g_\rho} \arrow{r}{d_e\act_{G_\rho}(\blank,v)} \arrow{d}{} & V_\rho \arrow{d}{} \arrow{r}{} & T_xY_\rho \arrow{d}{d_xi_\rho} \arrow{r}{} & 0 \\
			0 \arrow{r}{} & \mathfrak{g} \arrow{d}{} \arrow{r}{d_e\act_G(\blank,v)} & V \arrow{d}{} \arrow{r}{} & T_yY  \arrow{r}{} & 0 \\
			& \mathfrak{g}/\mathfrak{g}_\rho \arrow{d}{} \arrow{r}{} & V/V_\rho \arrow{d}{} & & \\
			& 0 & 0 \rlap{.}& &
		\end{tikzcd}
		\]
		The snake lemma then tells us that there exists a linear map $K \to \mathfrak{g}/\mathfrak{g}_\rho$ such that the sequence 
		\[
			0 \lra K \lra \mathfrak{g}/\mathfrak{g}_\rho \lra V/V_\rho
		\]
		is exact. In order to show that $K = 0$, it suffices to check that $\mathfrak{g}/\mathfrak{g}_\rho \to V/V_\rho$ is injective.

		The Lie algebra $\mathfrak{g}_\rho$ is given as follows. Recall that $G_\rho$ is the centralizer of the image of $\rho$. Therefore,
		\[
			\mathfrak{g}_\rho = \{ \xi \in \mathfrak{g} \mid \operatorname{Ad}_{\rho(t)}(\xi) = \xi \text{ for all } t \in \TT \}.
		\]

		Now let $\xi \in \mathfrak{g}$, and suppose that $\xi \cdot v \in V_\rho$.
                We show that this implies $\xi \in \mathfrak{g}_\rho$, which concludes the proof. To this end, let $t \in \TT$, and consider $\operatorname{Ad}_{\rho(t)}(\xi)$. We apply this element to $\rho(t)v$ and obtain
		\begin{align*}
			\operatorname{Ad}_{\rho(t)}(\xi)\cdot \rho(t)v 
			&= \rho(t)\cdot(\xi\cdot v) \\
			&= t.(\xi\cdot v) \\
			&= \xi \cdot (t.v) \\
			&= \xi \cdot (\rho(t)v)
		\end{align*}
		as both $v$ and $\xi\cdot v$ belong to $V_\rho$. This shows that $\operatorname{Ad}_{\rho(t)}(\xi)-\xi$ lies in the Lie algebra of the stabilizer $\operatorname{Stab}_G(\rho(t)v)$. By the stability of $\rho(t)v$ and Assumption~\ref{a:general_assumption}, the stabilizer is trivial. Hence we may conclude that $\operatorname{Ad}_{\rho(t)}(\xi) = \xi$, which finishes the proof.
	\end{proof}

	\begin{proof}[Proof of Theorem~\ref{t:closed_immersion}]
		We have established the three requirements of Proposition~\ref{p:osserman}. This proves Theorem~\ref{t:closed_immersion}.
	\end{proof}

	\section{Analysis of the fixed point locus} \label{s:fixedPointLocus}

	In this section, we will prove the main result of the paper, a description of the torus fixed point locus of the stable quotient.
	First, we are concerned with the index set of the decomposition of the fixed point locus into irreducible components.
	
	\begin{lem} \label{l:disjoint}
		Let $\rho_1, \rho_2\colon  \TT \to G$ be morphisms such that $F_{\rho_1} \cap F_{\rho_2} \neq \emptyset$. Then $\rho_1$ and $\rho_2$ are $G$-conjugate, and $F_{\rho_1} = F_{\rho_2}$.
	\end{lem}
	
	\begin{proof}
	  First of all, if $\rho_1$ and $\rho_2$ are $G$-conjugate, then it is immediate that $F_{\rho_1}$ and $F_{\rho_2}$ agree. 
          
		Concerning the other assertion, let $y \in F_{\rho_1} \cap F_{\rho_2}$. Then there exist $x_i \in V_{\rho_i} \cap V^{\st}(G,\theta)$ such that $\pi(x_i) = y$ for $i=1,2$. As $\pi$ is a geometric quotient, we find a $g \in G$ such that $x_2 = gx_1$. Then
		\[
			t.x_2 = \rho_2(t)x_2 = \rho_2(t)gx_1
		\]
		and on the other hand
		\[
			t.x_2 = t.gx_1 = gt.x_1 = g\rho_1(t)x_1. 
		\]
		This implies that $g^{-1}\rho_2(t)gx_1 = \rho_1(t)x_1$, and the stability of $x_1$ implies that $\rho_1(t) = g^{-1}\rho_2(t)g$ as the $G$-action on the stable locus is free.
	\end{proof}

	We fix a maximal torus $T \sub G$. Let $W = N_G(T)/T$ be the corresponding Weyl group.
	Now observe that for any $g \in G$, we get $F_{\rho_x} = F_{\rho_{gx}}$ and $\rho_{gx} = g\rho_xg^{-1}$. The image of the morphism $\rho_x$ is a torus; it is therefore contained in a maximal torus of $G$. As far as the sets $F_{\rho_x}$ are concerned, we may hence without loss of generality assume the image of $\rho_x$ to be contained in the fixed maximal torus $T \sub G$---for if not, we may replace $x$ with $gx$ for a suitable $g$ as all maximal tori are conjugate. By a slight abuse of notation, we also denote the morphism $\TT \to T$ by $\rho$. 
	
	\begin{lem} \label{l:conj}
		Let $\rho_1, \rho_2\colon  \TT \to T$ be morphisms. Assume that there exists an $h \in G$ such that $\rho_2 = h\rho_1h^{-1}$. Then there exists a $g \in N_G(T)$ such that $\rho_2 = g\rho_1g^{-1}$.
	\end{lem}
	
	\begin{proof}
		Let $s_1,s_2 \in T$, and suppose they are conjugate. Then, by \cite[Proposition~3.7.1]{Carter:93}, there exists a $g \in N_G(T)$ such that $s_2 = gs_1g^{-1}$. The book~\cite{Carter:93} deals with finite groups of Lie type, but the proposition holds true for every algebraic group which has a BN pair.
		
		Hence for every $t \in \TT$, there exists a $g_t \in N_G(T)$ such that $\rho_2(t) = g_t\rho_1(t)g_t^{-1}$. The element $g_t$ is unique up to $N_G(T) \cap C_G(\rho_1(t))$. Let $U = \{ t \in \TT \mid C_G(\rho_1(t)) = C_G(\rho_1) \}$. This is an open dense subset of $\TT$. Recall that $C_G(\rho_1) = G_{\rho_1}$. We obtain a map $U \to N_G(T)/(N_G(T) \cap G_{\rho_1})$. This map is locally constant, and the irreducibility of $U$ forces it to be constant, say of value $w$. We choose a lift $g$ of $w$ to $N_G(T)$. Then $\rho_2(t) = g\rho_1(t)g^{-1}$ for all $t$ in the dense subset $U$. This proves the lemma.
	\end{proof}

	We now state our main result, the description of the components of the fixed point locus of $V^{\st}(G,\theta)/G$ under the action of $\TT$.
	
	\begin{thm} \label{t:main}
		The $\TT\!$-fixed point locus $(V^{\st}(G,\theta)/G)^{\TT}\!$ decomposes into connected components $\smash{\bigsqcup_{\rho}} \smash{F_\rho}$ indexed by a full set of representatives of morphisms of tori $\rho\colon  \TT \to T$ up to conjugation with $W$. The component $F_\rho$ equals $\pi(V_\rho \cap V^{\st}(G,\theta))$, it is irreducible, and it is isomorphic to 
		\[
			V_\rho^{\st}(G_\rho,\theta)/G_\rho.
		\]
	\end{thm}
	
	\begin{proof}
		We have seen that every fixed point is contained in some subset $F_\rho$ which was defined as \mbox{$\pi(V_\rho \cap V^{\st}(G,\theta))$.} By Lemma~\ref{l:disjoint}, different conjugacy classes yield disjoint $F_\rho$. It suffices to consider morphisms $\rho\colon \TT \to T$ up to Weyl group action by Lemma~\ref{l:conj} and the reasoning before the lemma. By Theorem~\ref{t:stab_subgroup}, $F_\rho = \pi(\smash{V_\rho^{\st}(G_\rho,\theta)})$, and by Theorem~\ref{t:closed_immersion}, the morphism $i_\rho$ provides an isomorphism from $\smash{V_\rho^{\st}(G_\rho,\theta)/G_\rho}$ onto its image, which is $F_\rho$, and which is closed. A variety of the form $V_\rho^{\st}(G_\rho,\theta)/G_\rho$ is obviously irreducible.
	\end{proof}

	\section{A necessary condition for a non-empty fixed point component} \label{s:finiteness}
	
	The index set of the disjoint union in Theorem~\ref{t:main} is infinite; however, there are just finitely many $\rho$ for which $F_\rho$ is non-empty. It seems hard in this generality to give a necessary and sufficient condition for $F_\rho \neq \emptyset$. Still, it is possible to find a necessary condition which leaves us with just finitely many possibilities.
	
	Recall the notation $N_T(V)$ for the set of weights of the action of $T$ on $V$. We also want to introduce the symbol $\mathcal{N}_T(V)$ for the $\Q$-linear subspace of $X^*(T)_\Q$ which is spanned by $N_T(V)$.
	
	\begin{lem}\label{lem71}
		If there exists a vector $v \in V$ with a finite stabilizer, then $\mathcal{N}_T(V) = X^*(T)_\Q$.
	\end{lem}

	\begin{proof}
		Assume there are $\mu \in X^*(T) - (\mathcal{N}_T(V) \cap X^*(T))$. Let $N_T(V) = \{ \mu_1,\ldots,\mu_r \}$. Then we find a one-parameter subgroup $\lambda\colon \C^\times \to T$ such that $\langle \mu,\lambda \rangle \neq 0$ while $\langle \mu_i,\lambda \rangle = 0$ for $i=1,\ldots,r$. The subgroup $\lambda(\C^\times)$ acts trivially on $V$, and as $\lambda$ is non-trivial, this subgroup $\lambda(\C^\times)$ is infinite.
	\end{proof}

	In particular, we see that $\mathcal{N}_T(V) = X^*(T)_\Q$ is necessary for the existence of stable points for the $G$-action (with respect to any stability parameter $\theta$).
	
	\begin{prop} \label{p:finite}
		\leavevmode
		\begin{enumerate}
			\item Let $\rho\colon  \mathcal{T} \to T$ be a morphism. If\, $F_\rho \neq \emptyset$, then $\mathcal{N}_T(V_\rho) = X^*(T)_\Q$.
			\item The set $\{ \rho\colon  \TT \to T \text{ morphism } \mid \mathcal{N}_T(V_\rho) = X^*(T)_\Q \}$ is finite.
		\end{enumerate}

	\end{prop}
	
	\begin{proof}
		The first assertion follows from Lemma~\ref{lem71} and the characterization of $F_\rho$ as $V_\rho^{\st}(G_\rho,\theta)/G_\rho$. The second claim follows from the identity $N_T(V_\rho) = \{ \mu \in X^*(T) \mid (\mu,\rho^*(\mu)) \in N_{T \times \TT}(V) \}$ (see Remark~\ref{r:weights2}) and the following lemma.
	\end{proof}

	\begin{lem}
		Let $X$ and $Y$ be lattices. Let $E \sub X \times Y$ be a finite subset. For a $\Z$-linear map $f\colon  X \to Y$, consider $E_f := \{ x \in X \mid (x,f(x)) \in E \}$. Then the set
		\[
			\left\{ f \in \Hom_\Z(X,Y) \mid \langle E_f \rangle_\Q = X_\Q \right\}
		\]
		is finite.
	\end{lem}

	\begin{proof}
		Let $F$ be the set in question. Let $E_X$ and $E_Y$ be the images of $E$ under the projections $X \times Y \to X$ and $X \times Y \to Y$. Assume that $F \neq \emptyset$. Then $\langle E_X \rangle_\Q = X_\Q$.
		
		We equip $X_\R$ and $Y_\R$ with norms, both denoted by $\lVert \cdot \rVert$. For $f \in \Hom(X_\R,Y_\R)$, let $\norm{f} = \max_{\norm{x} \leq 1} \norm{f(x)}$ be the operator norm associated with these two norms. Define 
		\[
			R := \max_{y \in E_Y} \norm{y}. 
		\]
		Let $B \sub E_X$ be a subset which is an $\R$-basis of $X_\R$. Then we find a unique scalar product on $X_\R$ for which $B$ is an orthonormal basis. Let $\norm{\cdot}_B$ be the associated norm on $X_\R$, and let $\norm{f}_B = \max_{\norm{x}_B \leq 1} \norm{f(x)}$ be the corresponding operator norm on $\Hom(X_\R,Y_\R)$. As $\Hom(X_\R,Y_\R)$ is finite-dimensional, this norm is equivalent to the operator norm $\norm{\cdot}$ from above. We find $d_B, D_B > 0$ such that $d_B\norm{f} \leq \norm{f}_B \leq D_B\norm{f}$ for all $f \in \Hom(X_\R,Y_\R)$. Let $d>0$ be the minimum of all $d_B$, where $B$ ranges over all bases $B \sub E_X$.
		
		Now let $f \in F$. Then $\langle E_f \rangle_\R = X_\R$ and $f(E_f) \sub E_Y$. Choose a basis $B = \{x_1,\ldots,x_n\} \sub E_f$ of $X_\R$. Let $x = \sum \lambda_i x_i \in X$ be such that $\norm{x}_B \leq 1$. This implies $\smash{\norm{\lambda}_2^2} = \sum \lvert \lambda_i \rvert^2 \leq 1$, where $\lambda = (\lambda_1,\ldots,\lambda_n)$. It is well known that $\norm{\lambda}_1 \leq \sqrt{n}\norm{\lambda}_2$. Now
		\begin{align*}
			\norm{f(x)} \leq \underbrace{\sum_{i=1}^n \lvert\lambda_i\rvert}_{\leq \sqrt{n}} \underbrace{\norm{f(x_i)}}_{ \leq R} \leq R\sqrt{n}
		\end{align*}
		and therefore $\norm{f}_B \leq R\sqrt{n}$. This tells us $\norm{f} \leq \frac{1}{d_B}\norm{f}_B \leq \frac{R\sqrt{n}}{d}$. We have shown that the set $\{ f \in \Hom(X_\R,Y_\R) \mid \langle E_f \rangle_\R = X_\R \}$ is bounded. The set $F$ is the set of lattice points of this set and is therefore finite.
	\end{proof}

	\section{Application: Quiver moduli} \label{s:quiver_moduli}
	
	Before we discuss the general theory of torus actions on quiver moduli, let us look at a well-known example first.
	
	\begin{ex} \label{e:grass}
		Let $G = \GL_m(\C)$, which acts on $V = M_{m \times n}(\C)$, the space of $m \times n$ matrices, by left multiplication. Let $V_0 = \C^m$ be the natural representation of $\GL_m(\C)$. Then $V \cong V_0 \otimes \C^n$. 
		
		Assume that $m \leq n$. Let $\theta = \det \in X^*(G)$. The ring $\C[V]^{G,\det}$ is generated by the maximal minors of an $m \times n$ matrix. This shows that a matrix $A$ is $\theta$-semi-stable if and only if its rank is $m$. Every semi-stable point is therefore stable. The quotient $V^{\st}(G,\theta)/G$ identifies with the Grassmannian $\Gr_{n-m}(\C^n)$ by assigning to a matrix $A$ of full rank its kernel. Fix the maximal torus $T$ of diagonal matrices. The character lattice is $X^*(T) = \Z \epsilon_1 \oplus \ldots \oplus \Z \epsilon_m$, where $\epsilon_i(\diag(s_1,\ldots,s_m)) = s_i$. Note that the first fundamental weight is $\omega_1 = \epsilon_1$ and the natural representation $V_0$ has highest weight $\omega_1$. 
		
		Let $\TT = \C^\times$.
		We choose integers $w_1,\ldots,w_n$. We let $\TT$ act on $V$ by right multiplication with $\operatorname{diag}(t^{w_1},\ldots,t^{w_n})$. Let us assume that $w_1 \geq \ldots \geq w_n$, and let $0 = p_0 < p_1 < \ldots < p_k = n$ be such that
		\[
			w_1 = \dots = w_{p_1} > w_{p_1+1} = \dots = w_{p_2} > \dots > w_{p_{k-1}+1} = \dots = w_n.
		        \]
		Define $q_j = p_j-p_{j-1}$ for $j=1,\ldots,k$.
		
		We now consider the fixed point components. Let $\rho\colon  \C^\times  \to T$ be a one-parameter subgroup; it is given by $\rho(t) = \operatorname{diag}(t^{r_1},\ldots,t^{r_m})$. As we consider one-parameter subgroups up to conjugation by $W = S_m$, we may assume $r_1 \geq \ldots \geq r_m$. Let $0 = s_0 < s_1 < \ldots < s_l = m$ be such that
		\[
			r_1 = \dots = r_{s_1} > r_{s_1+1} = \dots = r_{s_2} > \dots > r_{s_{l-1}+1} = \dots = r_m, 
		\]
		and define $t_i = s_i-s_{i-1}$ for $i=1,\ldots,l$. We get $G_\rho = \GL_{t_1}(\C) \times \dots \times \GL_{t_l}(\C)$, viewed as block-diagonal matrices in $\GL_m(\C)$. The linear subspace $V_\rho \sub V$ by definition consists of all matrices $A = (a_{ij})$ for which right multiplication with $\diag(t^{w_1},\ldots,t^{w_n})$ agrees with left multiplication with $\diag(t^{r_1},\ldots,t^{r_m})$. This forces $a_{ij} = 0$ whenever $r_i \neq w_j$. Let $I \sub \{1,\ldots,l\}$ be the subset of those $i$ for which there exists a $j \in \{1,\ldots,k\}$ such that
		\[
			r_{s_i} = w_{p_j}.
		\]
		Such an index $j$ is unique. Let $I = \{i_1 < \dots < i_u \}$. The index corresponding to $i_\nu$ will be denoted by  $j_\nu$. We hence obtain an increasing sequence $j_1 < \dots < j_u$. As a $G_\rho$-representation, the subspace $V_\rho$ is isomorphic to
		\[
			\bigoplus_{\nu=1}^u M_{t_{i_\nu} \times q_{j_\nu}}(\C), 
		\]
		where $G_\rho$ acts on the $\supth{\nu}$ summand by left multiplication with the $\supth{i_\nu}$ factor in $G_\rho$. 
		
		For $F_\rho \neq \emptyset$, it is necessary by Proposition~\ref{p:finite}  that $\mathcal{N}_T(V_\rho) = X^*(T)_\Q$, which says that a matrix in $V_\rho$ must have no zero rows. Thus $I = \{1,\ldots,l\}$. (This tallies with the observation that an $m \times n$ matrix with rank $m$ cannot have zero rows.)		
		We get $i_\nu = \nu$ and $k \geq l$. A matrix $A \in V_\rho$ therefore has the following block shape: 
		\[
			\begin{pmatrix}
				\ldots & \boxed{t_1 \times q_{j_1}} & \ldots & \\
				& & \ldots & \boxed{t_2 \times q_{j_2}} & \ldots \\
				& & & & & \ddots \\
				& & & & & & \ldots & \boxed{t_l \times q_{j_l}} & \ldots \\
			\end{pmatrix}
		\]
		A matrix of this block shape can only have maximal rank if $q_{j_i} \geq t_i$ for all $i=1,\ldots,l$. If this is satisfied, then the fixed point component is, by Theorem~\ref{t:main}, isomorphic to
		\[
			F_\rho \cong \prod_{i=1}^l \Gr_{t_i}(\C^{q_{j_i}}).
		\]
		The one-parameter subgroups $\rho$ which occur are parametrized by the following finite set of data:
		\begin{itemize}
			\item a number $l \leq \min\{m,k\}$ and a sequence $0 = s_0 < s_1 < \dots < s_l = m$, 
			\item a sequence $1 \leq j_1 < \dots < j_l \leq k$
		\end{itemize}
		such that $t_i := s_i - s_{i-1} \leq q_{j_i}$ for every $i=1,\ldots,l$.
	\end{ex}

	Torus actions on quiver moduli contain Example~\ref{e:grass} as a special case. Let $Q$ be a quiver, \textit{i.e.}\ an oriented graph. Formally, $Q$ consists of two sets $Q_0$ and $Q_1$ and two maps $s,t\colon  Q_1 \to Q_0$. The set $Q_0$ is the set of vertices of $Q$, the set $Q_1$ is the set of arrows, and for $a \in Q_1$, the vertices $s(a)$ and $t(a)$ are the source and target of $a$, respectively. We assume that for any two vertices $i$ and $j$, there are just finitely many $a \in Q_1$ such that $s(a) = i$ and $t(a) = j$. A representation $M$ of $Q$ is a collection of complex vector spaces $(M_i)_{i \in Q_0}$ together with $\C$-linear maps $(M_a)_{a \in Q_1}$ such that $M_a\colon  M_{s(a)} \to M_{t(a)}$. If $\bigoplus_{i \in Q_0} M_i$ is finite-dimensional, we say that $M$ is finite-dimensional and call $\dimvect M := (\dim_\C M_i)_{i \in Q_0} \in \smash{\Z_{\geq 0}^{Q_0}}$ the dimension vector of $M$. Given two representations $M$ and $N$ of $Q$, a homomorphism $f\colon  M \to N$ is a collection $f_i\colon  M_i \to N_i$ of linear maps such that $N_af_{s(a)} = f_{t(a)}M_a$ holds for all $a \in Q_1$. In this way we obtain a $\C$-linear abelian category. We refer to \cite[Chapter~II]{ASS:06} for more details.
	
	Let $\Lambda(Q) = \{ \alpha = (\alpha_i)_{i \in Q_0} \in \Z^{Q_0} \mid \alpha_i = 0 \text{ for all but finitely many } i \}$, and let $\Lambda_+(Q) = \Lambda(Q) \cap \smash{\Z_{\geq 0}^{Q_0}}$.
	Fix a vector $\alpha \in \Lambda_+(Q)$. Fix complex vector spaces $V_i$ of dimension $\alpha_i$. Consider the (finite-dimensional) vector space
	\[
		V := R(Q,\alpha) := \bigoplus_{a \in Q_1} \Hom(V_{s(a)},V_{t(a)}).
	\]
	An element of $V$ can be regarded as a representation of $Q$ of dimension vector $\alpha$. The group $G_\alpha = \prod_{i \in Q_0} \GL(V_i)$ acts on $V$ by change of basis. Two elements of $V$ are isomorphic as representations if and only if they lie in the same $G_\alpha$-orbit. The central subgroup $\Delta := \C^\times \cdot \id$ acts trivially on $V$. Therefore, the action descends to an action of $G := PG_\alpha := G_\alpha/\Delta$.
	
	We will make a connection between weights of the action of a maximal torus of $G$ and the support $\supp(\alpha) = \{i \in Q_0 \mid \alpha_i \neq 0 \}$ of the dimension vector. To this end, fix a basis of every vector space $V_i$. Let $T_\alpha$ be the maximal torus of $G_\alpha$ of diagonal matrices with respect to our choices of bases. The image of $T_\alpha$ under $G_\alpha \to PG_\alpha = G$ is a maximal torus $T$ of $G$. For $i \in Q_0$ and $r \in \{1,\ldots,\alpha_i\}$, let $y_{i,r} \in X^*(T_\alpha)$ be the character which selects from $t = (t_i)_{i \in Q_0}$ the $\supth{r}$ diagonal entry of the matrix $t_i$. Then $X^*(T_\alpha)$ is the free abelian group generated by the basis elements $y_{i,r}$, and
	\[
		X^*(T) = \left\{ \sum_{i \in Q_0} \sum_{r=1}^{\alpha_i} a_{i,r}y_{i,r} \ \middle|\  \sum_{i \in Q_0} \sum_{r=1}^{\alpha_i} a_{i,r} = 0 \right\} \sub X^*(T_\alpha);
	\]
	it is generated as an abelian group by all differences $y_{j,s} - y_{i,r}$. Recall that $N_T(V)$ is the set of weights of the $T$-action on $V$ and $\mathcal{N}_T(V)$ is its $\Q$-linear span inside $X^*(T)_\Q$. As $G$ acts by conjugation, we have
	\[
		N_T(V) = \left\{ y_{t(a),s} - y_{s(a),r} \mid a \in Q_1,\ r \in \{1,\ldots,\alpha_{s(a)}\},\ s \in \{1,\ldots,\alpha_{t(a)} \} \right\}.
	\]

	\begin{lem} \label{l:connectedness}
		The full subquiver supported on $\supp(\alpha)$ is connected if and only if $\mathcal{N}_T(V) = X^*(T)_\Q$.
	\end{lem}
	
	\begin{proof}
		We may without loss of generality assume that $\supp(\alpha) = Q_0$. To show the direction from left to right, assume that $Q$ is connected. This means that for any two vertices $i$ and $j$, there exists an unoriented path which connects them. More precisely, there exist arrows $a_1,\ldots,a_k$ and signs $\epsilon_1,\ldots,\epsilon_k \in \{\pm 1\}$ such that
		\[
			s(a_1^{\epsilon_1}) = i, \quad
			t(a_1^{\epsilon_1}) = s(a_2^{\epsilon_2}), \quad
			\ldots, \quad
			t(a_{k-1}^{\epsilon_{k-1}}) = s(a_k^{\epsilon_k}), \quad
			t(a_k^{\epsilon_k}) = j.
		\]
		Here, $a^1 := a$ and $a^{-1}$ is a formal symbol for which we define $s(a^{-1}) = t(a)$ and $t(a^{-1}) = s(a)$. Given a character $y_{j,s} - y_{i,r}$, we choose an unoriented path connecting $i$ and $j$ as above and indices $r_1,\ldots,r_k, s_1,\ldots,s_k$ with $r_\nu \in \{1,\ldots,\smash{\alpha_{s(a_\nu^{\epsilon_\nu})}}\}$ and $s_\nu \in \{1,\ldots,\smash{\alpha_{t(a_\nu^{\epsilon_\nu})}}\}$ such that 
		\[
			r_1 = r, \quad
			s_1 = r_2, \quad
			\ldots, \quad
			s_{k-1} = r_k, \quad
			s_k = s.
		\]
		This yields
		\[
			y_{j,s} - y_{i,r} = \sum_{\nu=1}^k \left(y_{t(a_\nu^{\epsilon_\nu}),s_\nu} - y_{s(a_\nu^{\epsilon_\nu}),r_\nu}\right) \in \mathcal{N}_T(V).
		\]
		Suppose conversely that $Q$ is not connected. Then there exist two non-empty disjoint subsets $C,D \sub Q_0$ such that $C \sqcup D = Q_0$ and such that there are no arrows from one to the other. We let $\C^\times$ act on $V_i$ by weight $0$ if $i \in C$ and by weight $1$ if $i \in D$. This provides us with a morphism $\phi\colon  \C^\times \to T$; let $\phi^*\colon  X^*(T) \to \Z$ be the induced homomorphism on characters. If $i \in C$ and $j \in D$, then
		\[
			\phi^*(y_{j,s}-y_{i,r}) = 1
		\]
		for any $r \in \{1,\ldots,\alpha_i\}$ and $s \in \{1,\ldots,\alpha_j\}$. But as all arrows in $Q$ start and end in either $C$ or $D$, we see that $N_T(V) \sub \ker(\phi^*)$. Therefore, $y_{j,s} - y_{i,r} \notin \mathcal{N}_T(V)$.
	\end{proof}

	Let $\theta = (\theta_i)_{i \in Q_0} \in \Z^{Q_0}$ be such that $\theta^T\cdot \alpha = \sum_i \theta_i\alpha_i = 0$. We may interpret $\theta$ as a character of $G_\alpha$ by the assignment $g \mapsto \smash{\prod_{i \in Q_0}} \smash{\det(g_i)^{\theta_i}}$. The vanishing condition guarantees that it descends to a character of $G$; by a slight abuse of notation, we also denote this character  by $\theta$. By \cite[Proposition~3.1]{King:94} a representation $M \in V$ is semi-stable with respect to $\theta$ if and only if $\theta(U) \geq 0$ for all subrepresentations $U \sub M$; here $\theta(U)$ is by definition $\theta^T\cdot \dimvect U$. Moreover, $M$ is $\theta$-stable if and only if $\theta(U) > 0$ for all non-zero proper subrepresentations $U$ of $M$. As the $\theta$-stable representations are the simple objects in an appropriate abelian category (see \cite[Lemma~2.3]{Reineke:03}), Schur's lemma tells us that the automorphism group of a $\theta$-stable representation $M$ is $\C^\times$. Therefore, $G$ acts freely on the stable locus $V^{\st}(G,\theta)$. Assumption~\ref{a:general_assumption} is hence fulfilled. Set  
	\[
		M^{\theta\hypst}(Q,\alpha) := V^{\st}(G,\theta)/G.
	\]
	
	Now assume that $Q_0$ (and hence also $Q_1$) is a finite set.
	Let $\TT = (\C^\times)^{Q_1}$. Let a group element $t = (t_a)_{a \in Q_1} \in \TT$ act on $M = (M_a)_{a \in Q_1} \in V$ by 
	\[
		t.M = (t_aM_a)_{a \in Q_1}.
	\]
	It commutes with the $G$-action, so we obtain a $\TT$-action on $M^{\theta\hypst}(Q,\alpha)$.
	This action was studied by Weist in \cite{Weist:13}. We recall his result and show how it follows from our description.
	
	Let $\hat{Q}$ be the (infinite) quiver given by $\hat{Q}_i = Q_i \times X^*(\TT)$ for $i=0,1$ and $\hat{s},\hat{t}$ be  defined by
\[
		\hat{s}(a,\chi) = (s(a),\chi),\quad \hat{t}(a,\chi) = (t(a),\chi+x_a).
\]
	Here, $x_a\colon \TT \to \C^\times$ is the character given by $x_a(t) = t_a$. A vector $\beta \in \Lambda_+(\hat{Q})$ is called a \emph{cover} of $\alpha$ if $\sum_\chi \beta_{i,\chi} = \alpha_i$ for every $i \in Q_0$. On $\Lambda_+(Q)$ there is an action of the group $X^*(\TT)$ given by $\xi.\beta = (\beta_{i,\chi+\xi})_{i,\chi}$. We call two vectors translates of one another if they lie in the same orbit. We define 
	\[
		\hat{\theta} = \left(\hat{\theta}_{i,\chi}\right)_{(i,\chi) \in \hat{Q}_0} \in \Z^{\hat{Q}_0}
	\]
	by $\hat{\theta}_{i,\chi} = \theta_i$. The following is \cite[Theorem~3.8]{Weist:13}. 
	
	\begin{thm}[Weist]
		The fixed point locus $M^{\theta\hypst}(Q,\alpha)^{\TT}$ is the disjoint union $\bigsqcup_\beta F_\beta$ which ranges over all $\beta \in \Lambda_+(Q)$ which cover $\alpha$, up to translation. The component $F_\beta$ is isomorphic to
		\[
			M^{\hat{\theta}\hypst}(\hat{Q},\beta).
		\]
	\end{thm}
	
	The index set of the disjoint union is infinite, but if $F_\beta$ is non-empty, then the support of $\beta$ is connected. This reduces the consideration to a finite subset.
	
	Choose a basis of each of the vector spaces $V_i$. Let $T_\alpha \sub G_\alpha$ be the maximal torus of tuples of invertible diagonal matrices. Let $T = T_\alpha/\Delta$. It is a maximal torus of $G$. Let $\rho\colon  \TT \to T$ be a morphism. Then
	\[
		V_\rho = \{ M \in V \mid t_aM_a = \rho(t)_jM_a\rho(t)_i^{-1} \text{ for all } a\colon  i \to j \}.
	\]
	Via $\rho$, the torus $\TT$ operates on $V_i$, inducing a weight space decomposition $V_i = \smash{\sum_\chi} V_{i,\chi}$. The same argument as in \cite[Lemma~3.4]{Weist:13} shows that $M$ lies in $V_\rho$ if and only if we have $M_a(V_{i,\chi}) \sub V_{j,\chi+\epsilon_a}$. This shows that $V_\rho = R(\smash{\hat{Q}},\beta)$, where $\beta_{i,\chi} = \dim V_{i,\chi}$. The group $G_\beta$ is identified with the subgroup $\{ g \in G_\alpha \mid g_i(V_{i,\chi}) \sub V_{i,\chi} \text{ for all } i,\chi \}$ of $G_\alpha$. This is precisely the centralizer of $\rho$: Let $g \in G_\beta$, $t \in \TT$. To show that $g_i$ commutes with $\rho(t)_i$, we apply it to $v \in V_i$. Write $v$ as $v = \smash{\sum_\chi v_\chi}$ with $v_\chi \in V_{i,\chi}$. Then
	\begin{align*}
		g_i\rho(t)_iv &= \sum_\chi g_i\rho(t)_iv_\chi = \sum_\chi \chi(t)g_iv_\chi, \\
		\rho(t)_ig_iv &= \sum_\chi \rho(t)_i\underbrace{g_iv_\chi}_{\in V_{i,\chi}} = \sum_\chi \chi(t)g_iv_\chi.
	\end{align*}
	Conversely, suppose that $g \in G_\alpha$ commutes with $\rho$. Let $v \in V_{i,\chi}$. We want to show that $g_iv \in V_{i,\chi}$. Let $t \in \TT$, and compute
	\[
		\rho(t)_ig_iv = g_i\rho(t)_iv = g_i\chi(t)v = \chi(t)g_iv.
	\]
	We have shown that $V_\rho^{\st}(G_\rho,\theta)/G_\rho \cong M^{\hat{\theta}\hypst}(\hat{Q},\beta)$.
	
	Lemma~\ref{l:connectedness} shows that the support of $\beta$ is connected if and only if $\mathcal{N}_T(V_\rho) = X^*(T)_\Q$ holds.
	
	\begin{ex}
		Let $Q$ be the quiver
		\[\begin{tikzcd}
			1 & 2,
			\arrow["{a,b,c}", shift left=2, from=1-1, to=1-2]
			\arrow[shift right=2, from=1-1, to=1-2]
			\arrow[from=1-1, to=1-2]
		\end{tikzcd}\]
		the so-called $3$-Kronecker quiver. Fix the dimension vector $\alpha = (2,3)$. Then $V = M_{3 \times 2}(\C)^3$, and $G$ is the quotient of $G_\alpha = \GL_2(\C) \times \GL_3(\C)$ by the central one-parameter subgroup $\Delta = \{(t\id,t\id) \mid t \in \C^\times \}$. A representative $(g,h)$ of an element in $G$ acts on $(A,B,C) \in V$ by 
		\[
			(g,h)\cdot (A,B,C) = (hAg^{-1},hBg^{-1},hCg^{-1}).
		\]
		Let $\theta \in X^*(G)$ be the character defined by $\theta(g,h) = \det(g)^{-3}\det(h)^2$. An application of \cite[Proposition~3.1]{King:94} yields that $\theta$-semi-stability and $\theta$-stability agree and that a representation $(A,B,C)$ is $\theta$-stable if and only if $\dim \langle Ax,Bx,Cx \rangle \geq 2$ for every $x \in \C^2 \setminus \{0\}$. The $\theta$-stable moduli space $M^{\theta\hypst}(Q,\alpha) = V^{\st}(G,\theta)/G$ has dimension six.
		
		Let $\mathcal{T} = (\C^\times)^3$, which acts linearly on $V$ by 
		\[
			(t_a,t_b,t_c).(A,B,C) = (t_aA,t_bB,t_cC).
		\]
		We will describe the fixed points under this torus action.
		We fix the maximal torus $T_\alpha$ of $G_\alpha$ of invertible diagonal matrices and the induced maximal torus $T = T_\alpha/\Delta$ of $G$. The Weyl group is $W = S_2 \times S_3$. By Theorem~\ref{t:main} the $\mathcal{T}$-fixed point locus of the quotient $V^{\st}(G,\theta)/G$ is the disjoint union $\smash{\bigsqcup_\rho} F_\rho$, where $F_\rho = \smash{V_\rho^{\st}}(G_\rho,\theta)/G_\rho$; the union ranges over all morphisms $\rho\colon  \mathcal{T} \to T$  which satisfy $\mathcal{N}_T(V_\rho) = X^*(T)_\Q$, up to $W$-conjugation. There are three types of such morphisms:
		\begin{enumerate}
			\item The morphism $\rho\colon  \mathcal{T} \to T$ is defined by
			\[
				\rho(t_a,t_b,t_c) = \left( \begin{pmatrix} 1 & \\ & 1 \end{pmatrix}, \begin{pmatrix} t_a & & \\ & t_b & \\ & & t_c \end{pmatrix} \right).
			\]
			For $V_\rho$ and $G_\rho$, we obtain
			\[
				V_\rho = \left\{ \left( \begin{pmatrix} * & * \\ \phantom{*} & \phantom{*} \\ \phantom{*} & \phantom{*} \end{pmatrix}, \begin{pmatrix} \phantom{*} & \phantom{*} \\ * & * \\ \phantom{*} & \phantom{*} \end{pmatrix}, \begin{pmatrix} \phantom{*} & \phantom{*} \\ \phantom{*} & \phantom{*} \\ * & * \end{pmatrix} 
			   \right) \right\}, \quad
				G_\rho = \GL_2(\C) \times (\C^\times)^3. 
			\]
			The fixed point component $F_\rho = \smash{V_\rho^{\st}}(G_\rho,\theta)/G_\rho$ consists of a single point.
			\item For $m_1,\ldots,m_4 \in \{a,b,c\}$ with $m_1\neq m_2 \neq m_3 \neq m_4$, let $\rho\colon  \mathcal{T} \to T$ be defined by
			\[
				\rho(t_a,t_b,t_c) = \left( \begin{pmatrix} 1 & \\ & t_{m_2}t_{m_3}^{-1} \end{pmatrix}, \begin{pmatrix} t_{m_1} & & \\ & t_{m_2} & \\ & & t_{m_2}t_{m_3}^{-1}t_{m_4} \end{pmatrix} \right).
			\]
			The morphism associated with $(m_1,\ldots,m_4)$ is $W$-conjugate to the morphism associated with $(m_4,\ldots,m_1)$. Up to this symmetry, there are 12 such quadruples, namely
			\[
				abab,\ abac,\ abca,\ abcb,\ acab,\ acac,\ acbc,\ babc,\ bacb,\ bcac,\ bcbc,\ cabc.
			\]
			The associated morphisms are pairwise not $W$-conjugate.
			For all such $\rho$, $G_\rho = T$ holds. The subspaces $V_\rho$ are all four-dimensional. For instance for $m_1\ldots m_4 = abac$, we get 
			\[
				V_\rho = \left\{ \left( \begin{pmatrix} * & \phantom{*} \\ \phantom{*} & * \\ \phantom{*} & \phantom{*} \end{pmatrix}, \begin{pmatrix} \phantom{*} & \phantom{*} \\ * & \phantom{*} \\ \phantom{*} & \phantom{*} \end{pmatrix}, \begin{pmatrix} \phantom{*} & \phantom{*} \\ \phantom{*} & \phantom{*} \\ \phantom{*} & * \end{pmatrix} 
			   \right) \right\}.
			\]
			In all of the 12 cases, the fixed point component $F_\rho$ consists of a single point.
			\item For $m_1,\ldots,m_4 \in \{a,b,c\}$ such that $\{m_1,m_2,m_3\} = \{a,b,c\}$ and $m_4 \neq m_3$, let $\rho\colon  \mathcal{T} \to T$ be defined by
			\[
				\rho(t_a,t_b,t_c) = \left( \begin{pmatrix} 1 & \\ & t_{m_3}t_{m_4}^{-1} \end{pmatrix}, \begin{pmatrix} t_{m_1} & & \\ & t_{m_2} & \\ & & t_{m_3} \end{pmatrix} \right).
			\]
			The morphisms obtained from $(m_1,m_2,m_3,m_4)$ and $(m_2,m_1,m_3,m_4)$ are $W$-conjugate. Up to this symmetry we are left with the following six quadruples (which lead to pairwise non-$W$-conjugate morphisms):
			\[
				abca,\ abcb,\ acba,\ acbc,\ bcab,\ bcac.
			\]
			For all of those morphisms, $G_\rho = T$ holds. The subspaces $V_\rho$ are four-dimensional. For example for $abcb$, it looks as follows:
			\[
				V_\rho = \left\{ \left( \begin{pmatrix} * & \phantom{*} \\ \phantom{*} & \phantom{*} \\ \phantom{*} & \phantom{*} \end{pmatrix}, \begin{pmatrix} \phantom{*} & \phantom{*} \\ * & \phantom{*} \\ \phantom{*} & * \end{pmatrix}, \begin{pmatrix} \phantom{*} & \phantom{*} \\ \phantom{*} & \phantom{*} \\ * & \phantom{*} \end{pmatrix} 
			   \right) \right\}.
			\]
			However, although $\mathcal{N}_T(V_\rho) = X^*(T)_\Q$, the fixed point component for this morphism $\rho$, and all others of the third type, is empty because $V_\rho$ does not intersect $V^{\st}(G,\theta)$. Indeed, for $abcb$ the span $\langle Ae_2,Be_2,Ce_2 \rangle$ is one-dimensional for every $(A,B,C) \in V_\rho$.
		\end{enumerate}
		The above considerations show that $\mathcal{T}$ acts on $V^{\st}(G,\theta)/G$ with 13 isolated fixed points.
	\end{ex}

	\section{Application: Torus quotients} \label{s:toric}
	
	We are going to discuss quotients $V^{\st}(G,\theta)/G$ in the case when $G$ is a torus. In this case, $V^{\st}(G,\theta)/G$ is a toric variety.
	We will first describe the toric structure of $V^{\st}(G,\theta)/G$. We believe this should be well known, but we could not find it in the literature in the form we need it for our purposes. As a reference for toric geometry we use Fulton's book \cite{Fulton:93}.

	Let $V$ be a vector space of dimension $m$ and $G = (\C^\times)^r$ a torus which we assume to act linearly on~$V$. Let $a\colon  G \to \GL(V)$ be the induced morphism of algebraic groups. Decompose $V$ into weight spaces $V = \smash{\bigoplus_{\chi \in X^*(G)}} V_\chi$. Let $\chi_1,\ldots,\chi_p$ be the characters which occur in this decomposition, and let $e_{s,1},\ldots,e_{s,m_s}$ be a basis of $V_{\chi_s}$ for $s=1,\ldots,p$. Let $T$ be the maximal torus of $\GL(V)$ which corresponds to the basis $(e_{1,1},\ldots,e_{1,m_1},\ldots,e_{p,1},\ldots,e_{p,m_p})$ of $V$. Then $a$ factors through a morphism of tori $G \to T$, which we also call~$a$. Concretely, 
	\[
		a(g) = (\underbrace{\chi_1(g),\ldots,\chi_1(g)}_{m_1 \text{ times}},\ldots,\underbrace{\chi_p(g),\ldots,\chi_p(g)}_{m_p \text{ times}}).
	\]
	Let $x_{s,k} \in X^*(T)$ be the respective coordinate function.
	
	Let $\theta \in X^*(G)$. We describe the semi-stable and the stable locus with respect to $\theta$.
	Let 
	\[
		I = \{(s,k) \mid s \in \{1,\ldots,p\},\ k \in \{1,\ldots,m_s\} \}. 
	\]
	Write a vector $v \in V$ as $\smash{\sum_{(s,k) \in I}} v_{s,k}e_{s,k}$. We define $\supp(v) = \{(s,k) \in I \mid v_{s,k} \neq 0\}$. Let $\eta\colon  \C^\times  \to G$ be a one-parameter subgroup. Then
	\[
		\eta(z)v = \sum_{(s,k) \in I} z^{\langle \chi_s,\eta \rangle} v_{s,k}e_{s,k},
	\]
	so the limit $\lim_{z \to 0} \eta(z)v$ exists if and only if $\langle \eta,\chi_s \rangle \geq 0$ for all $(s,k) \in \supp(v)$. This is a condition which depends just on the support of $v$, so $\theta$-semi-stability and $\theta$-stability do as well. Let us be more precise. Let $\pr\colon  I \to \{1,\ldots,p\}$ be the projection to the first coordinate. For $S \sub I$, we define
	\[
		\tau_S := \cone(\chi_s)_{s \in \pr(S)} \sub X^*(G)_\R.
	\]
	It is a rational convex polyhedral cone. Its dual 
	\begin{align*}
		\tau_S^\vee &= \{y \in X_*(G)_\R \mid \langle x,y \rangle \geq 0 \text{ (all $x \in \tau_S$)} \} \\
			&= \{ y \in X_*(G)_\R \mid \langle \chi_s,y \rangle \geq 0 \text{ (all $s \in \pr(S)$)}\}
	\end{align*}
	is again a rational convex polyhedral cone. Denote by $\theta^\vee$ the dual of the ray through $\theta$, and let $\theta^+ = \{ y \in X_*(G) \mid \langle \theta,y \rangle > 0 \}$. The Hilbert--Mumford criterion tells us the following. 
	
	\begin{lem} \label{l:stability_support}
		Let $v \in V$ and $S := \supp(v)$.
		\begin{enumerate}
			\item The vector $v$ is $\theta$-semi-stable if and only if\, $\tau_S^\vee \sub \theta^\vee$.
			\item The vector $v$ is $\theta$-stable if and only if\, $\tau_S^\vee \setminus \{0\} \sub \theta^+$.
		\end{enumerate}
	\end{lem}
	
	We say a subset $S \sub I$ is $\theta$-\emph{semi-stable} if $\tau_S^\vee \sub \theta^\vee$, and we say $S$ is $\theta$-\emph{stable} if $\tau_S^\vee \setminus \{0\} \sub \theta^+$. The notions of (semi\nobreakdash-)stability depend only on $\pr(S)$. 
	
	\begin{lem} \label{l:stability_span}
		If $S \sub I$ is $\theta$-stable, then $(\chi_s)_{(s,k) \in S}$ spans $X^*(G)_\R$.
	\end{lem}
	
	\begin{proof}
		Assume that the $\R$-linear span $U$ of $(\chi_s)_{(s,k) \in S}$ is a proper subspace of $X^*(G)_\R$. Then there  exists a  $y \in X_*(G)\setminus \{0\}$ such that $\langle x,y \rangle = 0$ for all $x \in U$. Therefore, both $y$ and $-y$ lie in $\tau_S^\vee \setminus \{0\} \sub \theta^+$, which implies $\langle \theta,y \rangle > 0$ and $\langle \theta,-y \rangle > 0$, giving a contradiction.
	\end{proof}

	We call $S$ \emph{minimally} $\theta$-\emph{stable} if $S$ is $\theta$-stable and if no proper subset of $S$ is $\theta$-stable.
	
	\begin{lem} \label{l:minimal_stability}
		For a $\theta$-stable subset $S \sub I$, the following are equivalent:
		\begin{enumerate}
			\item The subset $S$ is minimally $\theta$-stable. 
			\item We have $\left|S\right| = r$. 
			\item The set $(\chi_s)_{(s,k) \in S}$ is a basis of\, $X^*(G)_\R$.
		\end{enumerate}
	\end{lem}

	\begin{proof}
		Let us assume $S$ is minimally stable, and let us show that $(\chi_s)_{(s,k) \in S}$ is a basis of $X^*(G)_\R$. By Lemma~\ref{l:stability_span}, it is enough to prove linear independence. Assume there are $(s',k') \in S$ such that $\chi_{s'}$ lies in the span of the $\chi_s$ with $(s,k) \in S' := S \setminus \{(s',k')\}$. Then $\tau_{S'}^\vee = \tau_S^\vee$, and therefore $S'$ is also stable. This contradicts the minimality of $S$.
		
		Now let $(\chi_s)_{(s,k) \in S}$ be a basis of $X^*(G)_\R$. As $\dim X(G)_\R = \dim G = r$, it is obvious that $\left|S\right| = r$.
		
		Finally, if $\left|S\right| = r$, then there can be no proper $\theta$-stable subset $S'$ of $S$ because the span of $(\chi_s)_{(s,k) \in S'}$ is a proper subspace of $X^*(G)_\R$.
	\end{proof}

	Lemma~\ref{l:stability_support} shows in particular that $V^{\sst}(G,\theta)$ and $V^{\st}(G,\theta)$ are $T$-invariant. The torus $T$ acts with a dense open orbit, isomorphic to $T$, on $V^{\st}(G,\theta)$ and $V^{\sst}(G,\theta)$, provided that these loci are non-empty. We obtain an action of $T$ on the semi-stable and stable quotient.
	
	We are going to assume in the following that $V^{\st}(G,\theta)$ is non-empty and that $G$ acts freely on the stable locus. Assumption~\ref{a:general_assumption} is hence fulfilled.
	
	Let $\mathcal{T}$ be the cokernel of $a\colon  G \to T$. It is a torus of rank $\dim V - \dim G$. We obtain a short exact sequence
	\[
		1 \lra G \xto{}{\;a\;} T \xto{}{\;\pi\;} \mathcal{T} \lra 1;
	\]
	note that Assumption~\ref{a:general_assumption} implies that $a$ is injective.
	The $T$-action on $V^{\sst}(G,\theta)/\!\!/G$ and on $V^{\st}(G,\theta)/G$ descends to an action of $\mathcal{T}$. The torus $\mathcal{T}$ acts on both with a dense open orbit isomorphic to $\mathcal{T}$. The quotients are therefore toric varieties. Let us describe the toric fan of $V^{\st}(G,\theta)/G$. 
	Let $M := X^*(\mathcal{T})$ and $N := X_*(\mathcal{T})$. Consider the short exact sequence of the lattices of one-parameter subgroups
	\[
		0 \lra X_*(G) \xto{}{\;a_*\;} X_*(T) \xto{}{\;\pi_*\;} N \lra 0.
	\]
	We have $X_*(T) = \bigoplus_{(s,k) \in I} \Z\epsilon_{s,k}$, where $\epsilon_{s,k}\colon  \C^\times \to T$ is defined by 
	\[
	\epsilon_{s,k}(z)e_{t,l} =
        \begin{cases} ze_{s,k} & \text{if } (s,k) = (t,l), \\
          e_{t,l} & \text{otherwise.} \end{cases}
	\]
	Note that the map $a_*\colon X_*(G) \to X_*(T)$ is given by $a_*(\eta) = \smash{\sum_{(s,k) \in I}} \langle \chi_s,\eta \rangle \epsilon_{s,k}$; we write $X_*(T)$ additively here. For a subset $J \sub I$, we define $\sigma_J := \cone(\pi_*\epsilon_{s,k})_{(s,k) \in J}$, a rational convex polyhedral cone in $N_\R$. We define
	\begin{align*}
		\Phi := \Phi^{\st}(\theta) &:= \{J \sub I \mid \tau_{J^c}^\vee \setminus \{0 \} \sub \theta^+ \} = \{J \sub I \mid J^c \text{ is $\theta$-stable}\}, \\
		\Delta := \Delta^{\st}(\theta) &:= \{ \sigma_J \mid J \in \Phi \}.
	\end{align*}
	In the above equation, $J^c = I \setminus J$.
	
	\begin{lem} \label{l:fan}
		Let $J \in \Phi$.
		\begin{enumerate}
			\item\label{l:fan-1} The elements $\pi_*\epsilon_{s,k}$ with $(s,k) \in J$ are linearly independent over $\R$. Therefore, $\sigma_J$ is a simplex of dimension~$\left|J\right|$.
			\item\label{l:fan-2} If\, $J' \sub J$, then $J' \in \Phi$ and $\sigma_{J'}$ is a face of $\sigma_J$. Every face of $\sigma_J$ is of this form.
			\item\label{l:fan-3} For $J_1,J_2 \in \Phi$, we get $\sigma_{J_1} \cap \sigma_{J_2} = \sigma_{J_1 \cap J_2}$.
		\end{enumerate}
	\end{lem}

	\begin{proof}
			\eqref{l:fan-1}~ Assume there are $\lambda = \smash{\sum_{(s,k) \in J}} a_{s,k} \epsilon_{s,k} \in X_*(T)$ with $\lambda \neq 0$ such that $\pi_*\lambda = 0$. Then there exists an $\eta \in X_*(G) \setminus \{0\}$ such that $\lambda = a_*\eta$. For all $(s,k) \notin J$, we have 
			\[
				\langle \chi_s,\eta \rangle = \langle a^*x_{s,k},\eta \rangle = \langle x_{s,k},\lambda \rangle = 0
			\]
			because $\lambda$ is supported on $J$. Therefore, $\langle \chi_s,\eta \rangle = 0$ for all $s \in \pr(J^c)$. In particular, $\eta \in \tau_{J^c}^\vee$. As $J \in \Phi$, we obtain $\langle \theta,\eta \rangle > 0$. But by the same argument, we also obtain $-\eta \in \tau_{J^c}^\vee$, which  implies $\langle \theta,-\eta \rangle > 0$, giving a contradiction.

                        \eqref{l:fan-2}~ The cone $\sigma_{J'}$ is the intersection of $\sigma_J$ with the supporting hyperplanes $(\pi_*\epsilon_{s,k})^\bot$, where $(s,k)$ ranges over $J\setminus J'$. Now let $\sigma$ be a face of $\sigma_J$. Then there exists a $u \in \sigma_J^\vee$ such that $\sigma = \sigma_J \cap u^\bot$. It is then easy to see that $\sigma = \sigma_{J'}$, where $J' = J \cap \supp(\pi^*u)$.
                        
			\eqref{l:fan-3}~ Assume there are
			\[
				\lambda_1 = \sum_{(s,k) \in J_1} a_{s,k}\epsilon_{s,k}, \quad \lambda_2 = \sum_{(s,k) \in J_2} b_{s,k}\epsilon_{s,k}
			\]
			with $a_{s,k},b_{s,k} \in \R_{\geq 0}$ such that $\lambda_1 \neq \lambda_2$ but $\pi_*\lambda_1 = \pi_*\lambda_2$. Then there  exists an $\eta \in X_*(G)$ such that $a_*\eta = \lambda_1-\lambda_2$. For all $(s,k) \in J_2^c$, we then have
			\[
				\langle \chi_s,\eta \rangle = \langle x_{s,k},\lambda_1-\lambda_2 \rangle = a_{s,k}-b_{s,k} = a_{s,k} \geq 0.
			\]
			As $J_2 \in \Phi$, this  implies $\langle \theta,\eta \rangle > 0$. But the same argument shows that we also have $\langle \theta,-\eta \rangle > 0$. We therefore again have a contradiction. 
	\end{proof}
	
	\begin{prop} \label{p:toric_fan_quotient}
		The toric fan of\, $V^{\st}(G,\theta)/G$ is $\Delta^{\st}(\theta)$.
	\end{prop}

	\begin{proof}
		For $J \in \Phi$, let us define $\sigma_J^+ = \smash{\cone(\epsilon_{s,k})_{(s,k) \in J}}$. The affine toric variety associated with the simplex $\smash{\sigma_J^+}$ is
		\[
			U_{\sigma_J^+} = \{ v \in V \mid J^c \sub \supp(v) \}.
		\]
		For $v \in \smash{U_{\sigma_J^+}}$ with $S := \supp(v)$, we get $\tau_S^\vee \setminus \{0\} \sub \tau_{J^c}^\vee \setminus \{0 \} \sub \theta^+$, so $v$ is $\theta$-stable by Lemma~\ref{l:stability_support}.
		We define $\Delta^+ = \{\sigma_J^+ \mid J \in \Phi \}$.
		We have argued that the toric variety $X_{\Delta^+}$ associated with the fan $\Delta^+$ is the stable locus $V^{\st}(G,\theta)$.
		It is enough to show that there is a map $\pi\colon  X_{\Delta^+} \to X_\Delta$ which is a categorical $G$-quotient. Consider the map
		\[
			\pi_J\colon  U_{\sigma_J^+} \lra U_{\sigma_J}
		\]
		induced by $\pi\colon  T \to \mathcal{T}$.
		By definition and an easy verification,
		\[
			U_{\sigma_J} = \Spec \C\left[M \cap \sigma_J^\vee\right] = \Spec \C\left[M \cap \left(\pi_*\sigma_J^+\right)^\vee\right] = \Spec \C\left[M \cap (\pi^*)^{-1}\left(\left(\sigma_J^+\right)^\vee\right)\right].
		\]
		The ring $\C[M \cap (\pi^*)^{-1}((\sigma_J^+)^\vee)]$ agrees with $\C[X^*(T) \cap (\sigma_J^+)^\vee]^G$, and therefore $\pi_J$ is a universal categorical $G$-quotient. By Lemma~\ref{l:fan}\eqref{l:fan-3}, the maps $\pi_J$ and $\pi_{J'}$ agree on the intersection and therefore glue to $\pi\colon  X_{\Delta^+} \to X_\Delta$ which, also by Lemma~\ref{l:fan}\eqref{l:fan-3}, is a categorical $G$-quotient. 
	\end{proof}

	We return to the action of $\mathcal{T}$. By general theory \cite[Section 3.1]{Fulton:93}, the torus $\mathcal{T}$ acts on $V^{\st}(G,\theta)/G = X_\Delta$ with finitely many orbits. The orbits are labeled by the faces of $\Delta$. In our case, the orbit corresponding to $\sigma_J$ is
	\begin{equation} \label{e:orbit} \tag{\#}
		O_{\sigma_J} = \{ \pi(v) \mid v \in V \text{ such that } \supp(v) = J^c \}/G.
	\end{equation}
	The dimension of $O_{\sigma_J}$ is $\dim V - \dim G - \dim \sigma_J = m - r - \left|J\right|$. The fixed points of the $\mathcal{T}\!$-action are therefore the orbits $O_{\sigma_J}$ associated with those $J \in \Phi$ for which $\left|J\right| = m-r$. By Lemma~\ref{l:minimal_stability}, a set $J$ with cardinality $m-r$ is maximal in $\Phi$, and every maximal member of $\Phi$ has cardinality $m-r$.
	
	We should like to apply our results to this setup and compare them with the description of the fixed points coming from the toric fan. In order to do so, we need $\mathcal{T}$ to act linearly on $V$. For this we choose, once and for all, a section $c\colon  \mathcal{T} \to T$ of $\pi\colon  T \to \mathcal{T}$. Then $\mathcal{T}$ acts on $V$ via $t.v = c(t)v$. This $\mathcal{T}\!$-action automatically commutes with the $G$-action, and the induced $\mathcal{T}\!$-action on $V^{\st}(G,\theta)/G$ agrees with the $\mathcal{T}\!$-action coming from the $T\!$-action on $V$.
	
	The group $G$ is already a torus. Let $\rho\colon  \mathcal{T} \to G$ be a morphism of algebraic groups. We get $G_\rho = G$ and
	\begin{align*}
		V_\rho &= \{ v \in V \mid t.v = \rho(t)\cdot v \text{ (all $t \in \mathcal{T}$)} \} \\
			&= \{ v \in V \mid c(t)v = a\rho(t)v \text{ (all $t \in \mathcal{T}$)} \}.
	\end{align*}
	For $(s,k) \in I$, let $c_{s,k}\colon  \mathcal{T} \to \C^\times$ be such that $c(t) = (c_{s,k}(t))_{(s,k) \in I} \in T = (\C^\times)^I$. Then $c(t)v = a\rho(t)v$ holds if and only if $c_{s,k}(t)v_{s,k} = \chi_s(\rho(t))v_{s,k}$ for all $(s,k) \in I$.
	Define the set $S_\rho \sub I$ as
	\[
		S_\rho = \{ (s,k) \in I \mid c_{s,k}(t) = \chi_s(\rho(t)) \text{ (all $t \in \mathcal{T}$)} \}.
	\]
	We see that $V_\rho^{\st}(G,\theta) \neq \emptyset$ if and only if $S_\rho$ is $\theta$-stable.
	
	\begin{prop} \label{p:bijection}
		Let $\rho\colon  \mathcal{T} \to G$ be a morphism of algebraic groups.
		\begin{enumerate}
			\item\label{p:bijection-1} If\, $S_\rho$ is $\theta$-stable, then $S_\rho$ is minimally $\theta$-stable.
			\item\label{p:bijection-2} If $\rho_1,\rho_2\colon  \mathcal{T} \to G$ are two morphisms of algebraic groups such that $S_{\rho_1} = S_{\rho_2}$, then $\rho_1 = \rho_2$.
			\item\label{p:bijection-3} For every minimally $\theta$-stable subset $S \sub I$, there exists a morphism $\rho\colon  \mathcal{T} \to G$ such that $S_\rho = S$.
		\end{enumerate}
	\end{prop}

	\begin{proof}
\eqref{p:bijection-1}~ Let $S := S_\rho$. By Lemma~\ref{l:minimal_stability}, it is enough to show that $\left|S\right| \leq r$. Let $U$ be the $\R$-linear span of $(x_{s,k})_{(s,k) \in S}$ inside $X^*(T)_\R$. This is an $\left|S\right|$-dimensional subspace. Consider the commutative diagram
			\[\begin{tikzcd}
				& {X^*(\mathcal{T})_\R} \\
				0 & {X^*(\mathcal{T})_\R} & {X^*(T)_\R} & {X^*(G)_\R} & 0
				\arrow[from=2-1, to=2-2]
				\arrow["{\pi^*}"', from=2-2, to=2-3]
				\arrow["{a^*}"', from=2-3, to=2-4]
				\arrow[from=2-4, to=2-5]
				\arrow["\id", from=2-2, to=1-2]
				\arrow["{c^*}"', from=2-3, to=1-2]
				\arrow["{\rho^*}"', bend right=20, from=2-4, to=1-2]
			\end{tikzcd}\]
			whose bottom row is exact. We know that $\dim U = \dim a^*(U) + \dim (U \cap \ker a^*)$. Let $u \in U \cap \ker a^* = U \cap \im \pi^*$; write it as $u = \pi^*v$. Then $v = c^*\pi^*v = c^*u = \rho^*a^*u = 0$ as $u \in \ker a^*$. Therefore, $U \cap \ker a^* = 0$ and thus 
			\[
				\left|S\right| = \dim U = \dim a^*(U) \leq \dim X^*(G)_\R = r.
			\]

                        \eqref{p:bijection-2}~ 	 Let $S := S_1 = S_2$. We have $\chi_s(\rho_1(t)) = c_{s,k}(t) = \chi_s(\rho_2(t))$ for all $(s,k) \in S$. Let $\epsilon_1,\ldots,\epsilon_r$ be a basis of $X_*(G)$. It is enough to show that $\rho_1\epsilon_i = \rho_2\epsilon_i$ for all $i=1,\ldots,r$. Consider $\eta_i \in X_*(G)$ defined by
			\[
				\eta_i(z) = \rho_1(\epsilon_i(z))\rho_2(\epsilon_i(z))^{-1}.
			\]
			Then $\langle \chi_s,\eta_i \rangle = 0$ for all $(s,k) \in S$. This implies that both $\eta_i$ and $-\eta_i$ belong to $\tau_S^\vee$. If $\eta_i$ is non-trivial, it  follows that $\langle \theta,\eta_i \rangle > 0$ and $\langle \theta,-\eta_i \rangle > 0$, giving a contradiction. We have shown that $\rho_1\epsilon_i = \rho_2\epsilon_i$.

                        \eqref{p:bijection-3}~ If $S$ is minimally $\theta$-stable, then $\smash{(\chi_s)_{(s,k) \in S}}$ is a basis of $X^*(G)_\R$ and therefore also a basis of $X^*(G)_\Q$. Consider the homomorphism of abelian groups $\Z^S = X^*((\C^\times)^S) \to X^*(G)$ which maps the basis vector $x_{s,k}$ to $\chi_s$, \textit{i.e.}\ the restriction of $a^*$. It induces a homomorphism of bialgebras
			\[
				\C[\Z^S] = \C[(\C^\times)^S] \lra \C[X^*(G)] = \C[G].
			\]
			This yields a morphism of algebraic groups, which on closed points is given as
			\[
				a_S\colon  G \lra (\C^\times)^S,\quad g \longmapsto (\chi_s(g))_{(s,k) \in S}.
			\]
			As $\C[(\C^\times)^S] \to \C[G]$ is injective, $a_S$ is dominant. Moreover, because $(\chi_s)_{(s,k) \in S}$ generates $X^*(G)_\Q$, the morphism $a_S$ is finite. This implies that $a_S$ is surjective.
			By our Assumption~\ref{a:general_assumption}, $G$ acts freely on the non-empty set
			\[
				\{v \in V^{\st}(G,\theta) \mid \supp(v) = S \}.
			\]
			As a consequence, $a_S$ is also injective. In characteristic zero, a bijective morphism of algebraic groups is an isomorphism. We define
			\[
				\rho := a_S^{-1}\pr_Sc, 
			\]
			where $\pr_S\colon  T \to (\C^\times)^S$ is the projection. Then
			\[
				a_S\rho = a_Sa_S^{-1}\pr_Sc = \pr_Sc,
			\]
			which means $\chi_s(\rho(t)) = c_{s,k}(t)$ for all $(s,k) \in S$. We have shown that $S \sub S_\rho$. But $S$ is $\theta$-stable, so $S_\rho$ is as well. By~\eqref{p:bijection-1}, it is then minimally $\theta$-stable, which forces $S = S_\rho$. 
	\end{proof}

	Proposition~\ref{p:bijection} shows that the description of the fixed points using the toric fan agrees with our description. More precisely, we have the following.  
		
	\begin{cor}
		The assignment $\rho \mapsto S_\rho^c$ defines a bijection between the set
		\[
			\{ \rho\colon  \mathcal{T} \to G \mid \rho \text{ morphism of algebraic groups such that } V_\rho^{\st}(G,\theta) \neq \emptyset \}
		\]
		and the set of maximal members of\, $\Phi^{\st}(\theta)$. Moreover, for every $\rho\colon  \mathcal{T} \to G$ for which $V_\rho^{\st}(G,\theta) \neq \emptyset$, let $J := S_\rho^c$. Then the subsets $V_\rho^{\st}(G,\theta)/G$ and $O_{\sigma_J}$ of\, $V^{\st}(G,\theta)/G$ agree and consist of a single point.
	\end{cor}

	\begin{proof}
		The bijectivity of $\rho \mapsto S_\rho^c$ was already established in Proposition~\ref{p:bijection}. The second statement follows from the description (\ref{e:orbit}) of $O_{\sigma_J}$ and from
	\begin{equation*}\pushQED{\qed}	
			V_\rho^{\st}(G,\theta) = \{ v \in V \mid \supp(v) = S_\rho \}.
\qedhere \popQED
	\end{equation*}
\renewcommand{\qed}{}     
	\end{proof}    
	
	We finish this section, and the paper, with an example.

	\begin{ex}
		Fix a number $d \in \Z_{\geq 0}$. Consider the vector space $V = \C^4$ with the action of $G = (\C^\times)^2$ given by
		\[
			(g,h)\cdot (x_0,x_1,y,z) = (gx_0,gx_1,hy,g^dhz).
		\]
		This means, as a $G$-representation, $V = \C(1,0)^2 \oplus \C(0,1) \oplus \C(d,1)$. Let $\theta\colon  G \to \C^\times$ be the character given by $\theta(g,h) = g^{d+1}h$. Using the Hilbert--Mumford criterion, we can show that
		\[
			V^{\sst}(G,\theta) = V^{\st}(G,\theta) = \{ v = (x_0,x_1,y,z) \mid (x_0,x_1) \neq (0,0) \neq (y,z) \}.
		\]
		We see that the quotient $V^{\st}(G,\theta)/G$ is isomorphic to $\P(\OO_{\P^1} \oplus \OO_{\P^1}(d))$, the $\supth{d}$ Hirzebruch surface.
		
		Observe that $G$ acts freely on the stable locus. So Assumption~\ref{a:general_assumption} is fulfilled. We get $T = (\C^\times)^4$, and the morphism $a\colon  G \to T$ is given by
		\[
			a(g,h) = (g,g,h,g^dh).
		\]
		The morphism $\pi\colon  T \to \mathcal{T} = (\C^\times)^2$ defined by $\pi(z_1,\ldots,z_4) = (z_1z_2^{-1},z_3z_2^dz_4^{-1})$ is a cokernel of $a$. Let $I = \{1,\ldots,4\}$ and $\epsilon_1,\ldots,\epsilon_4$ be the basis of $X_*(T)$ consisting of the embeddings in the $\supth{i}$ diagonal entry. We read off
		\[
			\Phi = \{ J \sub \{1,\ldots,4\} \mid \{1,2\} \nsubseteq J \nsupseteq \{3,4\} \},
		\]
		and  by Proposition~\ref{p:toric_fan_quotient}, the toric fan is $\Delta = \{ \cone(\pi^*\epsilon_i)_{i \in J} \mid J \in \Phi \}$. The maximal elements of $\Phi$ are $\{1,3\}$, $\{2,3\}$, $\{1,4\}$, and $\{2,4\}$. They correspond to the fixed points 
		\[
			\pi(0,1,0,1),\quad \pi(1,0,0,1),\quad \pi(0,1,1,0),\quad \pi(1,0,1,0).
		\]
		A section of $\pi$ is given by $c\colon  \mathcal{T} \to T$, which sends $(t_1,t_2)$ to $(t_1,1,t_2,1)$. For $\rho\colon  \mathcal{T} \to G$, we obtain
		\[
			V_\rho = \{ v \in V \mid c(t)v = a\rho(t)v \text{ (all $t \in \mathcal{T}$)} \}.
		\]
		Let $\rho(t) = (t_1^{m_1}t_2^{m_2},t_1^{n_1}t_2^{n_2})$. Then for $v = (v_1,\ldots,v_4)$, the condition $c(t)v = a\rho(t)v$ is equivalent to the four equations
		\begin{align*}
			t_1v_1 &= t_1^{m_1}t_2^{m_2}v_1, \\
			v_2 &= t_1^{m_1}t_2^{m_2}v_2, \\
			t_2v_3 &= t_1^{n_1}t_2^{n_2}v_3, \\
			v_4 &= t_1^{dm_1+n_1}t_2^{dm_2+n_2}v_4.
		\end{align*}
		There are four possibilities for $\rho$ which admit a non-empty stable locus $V_\rho^{\st}(G,\theta)$. They are $\rho_1(t) = (t_1,t_2)$, $\rho_2(t) = (1,t_2)$, $\rho_3(t) = (t_1,\smash{t_1^{-d}})$, and $\rho_4(t) = (1,1)$. They give the fixed points
		\[
			\pi(1,0,1,0),\quad \pi(0,1,1,0),\quad \pi(1,0,0,1),\quad \pi(0,1,0,1).
		\]
	\end{ex}



\begin{thebibliography}{BGM19+++}

\bibitem[ASS06]{ASS:06}
I.~Assem, D.~Simson, and A.~Skowro{\'n}ski, \emph{Elements of the
  representation theory of associative algebras. {V}ol.~1}, London Math.\ Soc.\ Stud.\  Texts, vol.~65, Cambridge Univ.\ Press, Cambridge, 2006,

\bibitem[BGM19]{BGM:19}
M.~Bate, H.~Geranios, and B.~Martin, \emph{Orbit closures and invariants},
  Math.~Z.\ \textbf{293} (2019), no.~3-4, 1121--1159.

\bibitem[Bri10]{Brion:10}
M.~Brion, \emph{Introduction to actions of algebraic groups}, Les cours du CIRM
  \textbf{1} (2010), no.~1, 1--22.
  
\bibitem[Car93]{Carter:93}
R.\,W.~Carter, \emph{Finite groups of lie type. Conjugacy classes and
complex characters} (reprint of the 1985 original), Wiley Classics
Lib., Wiley-Intersci.\ Publ., John Wiley {\&} Sons, Ltd., Chichester,
1993.

\bibitem[Cox97]{Cox:97}
D.\,A.~Cox, \emph{Recent developments in toric geometry}, in: \emph{Algebraic
  geometry} (Santa Cruz, 1995), pp.~389--436, Proc.\ Sympos.\ Pure Math., vol.~62, Amer.\  Math.\ Soc., Providence, RI, 1997.

\bibitem[Dol03]{Dolgachev:03}
I.~Dolgachev, \emph{Lectures on invariant theory}, London Math.\ Soc.\ Lecture Note Ser., vol.~296, Cambridge Univ.\ Press, Cambridge, 2003.

\bibitem[ES89]{ES:89}
G.~Ellingsrud and S.\,A.~Str{\o}mme, \emph{On the {C}how ring of a geometric
  quotient}, Ann.\ of Math.~(2) \textbf{130} (1989), no.~1, 159--187.

\bibitem[Ful93]{Fulton:93}
W.~Fulton, \emph{Introduction to toric varieties}, Ann.\ of Math.\ Stud., vol.~131, William Roever Lectures Geom., Princeton Univ.\ Press, Princeton, NJ, 1993. 

\bibitem[Hal10]{Halic:10}
M.~Halic, \emph{Cohomological properties of invariant quotients of affine
  spaces}, J.~Lond.\ Math.\ Soc.~(2) \textbf{82} (2010), no.~2, 376--394.

\bibitem[Hes78]{Hesselink:78}
W.\,H.~Hesselink, \emph{Uniform instability in reductive groups}, J.~reine
  angew.\ Math.\ \textbf{303/304} (1978), 74--96.

\bibitem[Hos14]{Hoskins:14}
V.~Hoskins, \emph{Stratifications associated to reductive group actions on
  affine spaces}, Q.~J.~Math.\ \textbf{65} (2014), no.~3, 1011--1047.

\bibitem[Kem78]{Kempf:78}
G.\,R.~Kempf, \emph{Instability in invariant theory}, Ann.\ of Math.~(2)
  \textbf{108} (1978), no.~2, 299--316.

\bibitem[Kem18]{Kempf:18}
\bysame, \emph{Instability in invariant theory} (transcribed by I.~Morrison), preprint, \arXiv{1807.02890} (2018).

\bibitem[Kin94]{King:94}
A.\,D.~King, \emph{Moduli of representations of finite-dimensional algebras},
  Quart.~J.\ Math.\ Oxford Ser.~(2) \textbf{45} (1994), no.~180, 515--530.

\bibitem[Kir84]{Kirwan:84}
F.\,C.~Kirwan, \emph{Cohomology of quotients in symplectic and algebraic
  geometry}, Math.\ Notes, vol.~31, Princeton Univ.\ Press,
  Princeton, NJ, 1984.

\bibitem[Kra84]{Kraft:84}
H.~Kraft, \emph{Geometrische {M}ethoden in der {I}nvariantentheorie}, Aspects
  Math., D1, Friedr. Vieweg \& Sohn, Braunschweig, 1984.

\bibitem[Lun73]{Luna:73}
  D.~Luna, \emph{Slices \'{e}tales}, in: \emph{Sur les groupes alg\'{e}briques},
pp.~81--105, Bull.\ Soc.\ Math.\ France M\'{e}m. \textbf{33} (1973). 
 
\bibitem[Lun75]{Luna:75}
\bysame, \emph{Adh\'{e}rences d'orbite et invariants}, Invent.\ Math.\
  \textbf{29} (1975), no.~3, 231--238.

\bibitem[MFK94]{GIT:94}
  D.~Mumford, J.~Fogarty, and F.~Kirwan, \emph{Geometric invariant theory}, 3rd ed., Ergeb.\ Math.\ Grenzgeb.~(2), vol.~34,
  Springer-Verlag, Berlin, 1994.

\bibitem[Moz23]{Mozgovoy:23}
S.~Mozgovoy, \emph{Translation quiver varieties}, J.~Pure Appl.\ Algebra
  \textbf{227} (2023), no.~1, 107156.

\bibitem[Nes84]{Ness:84}
L.~Ness, \emph{A stratification of the null cone via the moment map} (with an appendix by D.~Mumford),  Amer.~J.\  Math.\ \textbf{106} (1984), no.~6, 1281--1329. 


\bibitem[Rei03]{Reineke:03}
M.~Reineke, \emph{The {H}arder-{N}arasimhan system in quantum groups and
  cohomology of quiver moduli}, Invent.\ Math.\ \textbf{152} (2003), no.~2,
  349--368.

\bibitem[Rei08]{Reineke:08}
\bysame, \emph{Moduli of representations of quivers}, in: \emph{Trends in representation  theory of algebras and related topics}, pp.~589--637, EMS Ser.\ Congr.\ Rep., Eur.\ Math.\ Soc.,  Z\"urich, 2008.

\bibitem[Wei13]{Weist:13}
T.~Weist, \emph{Localization in quiver moduli spaces}, Represent.\ Theory
\textbf{17} (2013), 382--425.


\end{thebibliography}
\end{document}